\documentclass[12pt]{amsart}
\usepackage{geometry}                
\usepackage{graphicx}

\usepackage{amssymb}
\usepackage{hyperref}
\usepackage{float}
\usepackage{caption}
\usepackage{subcaption}

\theoremstyle{plain} 
\newtheorem{theorem}{Theorem}[section]
\newtheorem{lemma}[theorem]{Lemma}

\theoremstyle{definition} 
\newtheorem{definition}[theorem]{Definition}
\newtheorem{remark}[theorem]{Remark}

\newcommand{\Z}{\mathbb{Z}}
\newcommand{\R}{\mathbb{R}}
\newcommand{\C}{\mathbb{C}}
\newcommand{\N}{\mathbb{N}}
\newcommand{\s}{\mathbb{S}}
\newcommand{\h}{\mathbb{H}}

\newcommand{\so}{\mbox{\bf SO}(3)}

\newcommand{\re}{\operatorname{Re}}

\newcommand{\Res}{\operatorname{Res}}

\def\bj{Bj\"o{}rling }
\def\we{Weierstrass }
\def\mo{M\"o{}bius }

\begin{document}

\begin{title}
{Explicit Bj\"o{}rling Surfaces with Prescribed Geometry}
\end{title}

\author
{Rafael L\'o{}pez}
\address{Rafael L\'o{}pez\\Departamento de Geometría y Topolog\'i{}a\\ Instituto de Matem\'aticas (IEMath-GR)\\
Universidad de Granada\\
18071 Granada
\\Spain}
\author{
Matthias Weber
}
\address{Matthias Weber\\Department of Mathematics\\Indiana University\\
Bloomington, IN 47405
\\USA}
\thanks{The  first  author was  partially supported by  the MINECO/FEDER grant MTM2014-52368-P. The second author  was partially supported by a grant from the Simons Foundation (246039 to Matthias Weber)}

\subjclass[2010]{Primary 53A10, 53C43; Secondary 53C45}
\date{\today}
\maketitle

\begin{abstract}
We develop a new method to construct explicit, regular minimal surfaces in Euclidean space that are defined on the entire complex plane with controlled geometry. More precisely we show that for a large class of planar curves $(x(t), y(t))$ one can find a third coordinate $z(t)$ and normal fields $n(t)$ along the space curve $c(t)=(x(t), y(t), z(t))$ so that the \bj formula applied to $c(t)$ and $n(t)$ can be explicitly evaluated. We give many  examples.
 \end{abstract}

\section{Introduction}

In several recent papers (eg \cite{bk1,dean,howhi1,mewe1,mi1}) embedded minimal disks have been constructed that have the appearance of a {\em coil}. They contain a core curve along which the surface normal rotates in a controllable way. Increasing the rotational speed of the normal allows then to construct and study minimal limit foliations and their singular sets.
The classical example  is the helicoid. Here, increasing the (constant) rotational speed is equivalent to scaling the surface.

The next complicated case of minimal M\"o{}bius strips with core curve a circle and a normal field that rotates with constant speed was studied by Mira (\cite{mi1}).
 
More recently, and Meeks and the second author (\cite{mewe1}) have generalized this construction to surfaces with core curve any given compact $C^1$ curve. Instrumental for this construction was the possibility to explicitly control the model case of a circular core curve. A surprising byproduct of this investigation was that the ``circular helicoids'' where not only defined near the core circle, but were in fact finite total curvature minimal surfaces.

Our approach to construct new explicit  and global examples of minimal surfaces  where the normal rotates arbitrarily fast about the core curve utilizes the \bj formula.
This is an integral formula that produces a minimal surface for any given real analytic space curve $c$ and unit normal field $n$ along $c$. Our first problem  is that the
integrals arising in this formula are rarely explicit. Using quaternions, we overcome this difficulty by constructing suitable curves in $\so$ that serve as frame fields. While this alone gives us a plethora of new examples, we face a second problem: We would like (to some extent) control the geometry of the constructed surfaces. 

This is achieved in the second part of the paper, where we show that a large class of planar curves (containing many classical curves) admit lifts into Euclidean space that can be used as core curves for explicit minimal coils. For closed planar curves, the lifted curve will be periodic, and its translational period can be controlled by a parameter in the construction. We show that for a generic choice of the parameter, the surface is defined in the entire complex plane and regular everywhere.

In the last section we give many new examples.  For instance, the method is powerful enough to create an explicit knotted minimal \mo band of finite total curvature.

\section{Explicit \bj surfaces}\label{sec:explicit}

We begin by reviewing the \bj formula.
Let $c:I\subset\R\rightarrow\R^3$ be a real analytic curve, called the {\em core curve}, and let $n$ be a real analytic unit vector field along $c$ with $\langle c'(t),n(t)\rangle=0$ for every $t\in I$. By analyticity, the functions $c$ and $n$ have holomorphic extensions $c(z)$ and $n(z)$ to a simply-connected domain $\Omega$ with $I \subset\Omega$. Fix $t_0\in I$ and define
\begin{equation}\label{bj}
X(u,v)=X(z)=\re \left(c(z)-i\int_{t_0}^z n(w)\wedge c'(w)\ dw \right),\quad\quad z=u+iv \ .
\end{equation}

We point out that the integral in \eqref{bj} is taken along an arbitrary path in $\Omega$ joining $t_0$ and $z$ and it does not depend on the chosen path because $\Omega$ is simply-connected. The  surface $X(u,v)$  is the unique minimal surface such that  the curve $c$ is the parameter curve $v=0$ and the unit normal field to the surface $X(u,v)$ coincides with $n$ along $c$ (\cite{dhkw}). We say that $X(u,v)$ is the \bj surface with \bj data $\{c,n\}$.  

As the parametrization given by the \bj formula is conformal, one can always find \we data $G$ and $dh$ such that
\begin{equation}\label{eqn:wei}
c' - i\cdot  n\wedge c' = 
\begin{pmatrix}
\frac12\left(\frac1G-G\right)\, dh \\
\frac{i}2\left(\frac1G+G\right)\, dh \\
dh\\
\end{pmatrix} \ .
\end{equation}
Here $G$ denotes the stereographic projection of the Gauss map and $dh$ the height differential as usual.
In terms of these \we data, the conformal factor of the Riemannian metric of $X(u,v)$ is given by \cite{ka5}
\[
|dX| = \frac12\left(|G| +\frac1{|G|}\right) |dh| \ .
\]
This will allow us to determine when our surfaces are regular.

In order to construct minimal surfaces with arbitrarily fast rotating normal, we would like to begin with a real analytic space curve $c$ and two unit normal fields $n_1(t)$, $n_2(t)$ such that $c'(t)$, $n_1(t)$ and $n_2(t)$ are an orthogonal basis of $\R^3$ for each $t$. Then we  form a spinning normal relative to $n_1$ and $n_2$ by writing
\[
n(t) = \cos(\alpha(t))\cdot n_1(t) + \sin(\alpha(t) )\cdot n_2(t)
\]
for a suitable rotation angle function $\alpha(t)$, and use $c$ and $n$ as \bj data.

Our construction of  globally defined and explicit examples is based on the following idea. If we could choose $c$, $n_1$ and $n_2$ such that the matrix $(c'(t), n_1(t), n_2(t))$ is a curve in $\so$ with entries given as trigonometric polynomials, then the integral in the \bj formula can be explicitly evaluated for any linear function $\alpha(t) = at+b$. Note that both the helicoid and the circular helicoid in \cite{mewe1} are of this form.

We can in fact do somewhat better than that, both relaxing the requirements on the matrix entries and on the matrix itself.
We begin by formalizing which functions we allow as coordinate functions.

\begin{definition}
We call a real valued function of a real variable $t$ {\em polyexp} if it is a linear combination  of functions of the form $t^n e^{k t}$,
where $n\in\N$ and $k\in\C$.  
\end{definition}

Hence polynomials, exponentials, and trigonometric functions are all polyexp. Using integration by parts and induction, we obtain the following simple observation.

\begin{lemma}
The products, integrals and derivatives of polyexp functions are again polyexp.
\end{lemma}

In fact, based on the formulas that follow, we could allow any class of analytic functions that is closed under sums, products, derivatives, and integration. For instance, if one is not interested in the explicit nature of new examples but rather in their global features, one could allow all entire functions that are real valued on the real axis.

As a consequence of this definition, we obtain the following corollary, which is the basis for our construction.

\begin{theorem}\label{thm:key}
Given  polyexp  vectors $e_1(t)$, $e_2(t)$, $e_3(t)\in \R^3$ and a nonvanishing polyexp function $\mu(t)$ such that $\frac 1{\mu} (e_1, e_2, e_3)\in \so$, we can explicitly find a  curve $c(t)$ with $c'(t) = e_1(t)$. Then, the curve $c(t)$ and the rotating normal 
\[
n(t)  =\frac 1{\mu(t)} \left(  \cos(at+b) e_2(t) + \sin(a t+b) e_3(t)\right)
\]
provide \bj data that can be explicitly integrated. Moreover, the resulting \bj surface is defined on the entire complex plane.
\end{theorem}
\begin{proof}
Note that in the \bj formula, the integrand $c' \wedge n$ is polyexp because the factor $\mu$ cancels.
\end{proof}

\section{The Quaternion Method}

To apply the method from the previous section, we need to produce examples of polyexp curves in $\R\cdot \so$. Our first approach utilizes quaternions.

Let $\h$ denote the real algebra of quaternions that we write as usual as $q_1{\bf 1}+q_2{\bf i}+q_3{\bf j}+q_4{\bf k}$, where $\{{\bf 1},{\bf i}, {\bf j}, {\bf k}\}$ is the canonical basis of $\h$ and $(q_1,q_2,q_3,q_4)\in\R^4$. 

For non-zero $q\in\h$, the linear map

\begin{align*}
\rho_q:\h & \rightarrow\h \\
\rho_q(v)&\mapsto qv \bar{q}
\end{align*}

acts on the imaginary quaternions as an element   $\Phi(q)\in \R\cdot \so$. Explicitly, 
\[
\Phi(q_1,q_2,q_3,q_4) = 
\begin{pmatrix}
q_1^2+q_2^2-q_3^2-q_4^2 & 2 q_2 q_3-2 q_1 q_4 & 2 q_1 q_3+2 q_2 q_4 \\
 2 q_1 q_4+2 q_2 q_3 & q_1^2-q_2^2+q_3^2-q_4^2 & 2 q_3 q_4-2 q_1 q_2 \\
 2 q_2 q_4-2 q_1 q_3 & 2 q_1 q_2+2 q_3 q_4 & q_1^2-q_2^2-q_3^2+q_4^2 \\
\end{pmatrix} \ .
\]

As a consequence we have

\begin{lemma}
Let $q(t)$ be a polyexp curve in $\R^4$. Let $\mu(t) = |q(t)|^2$. Then both $\mu$ and $\Phi(q)$ are polyexp, and $\frac1\mu \Phi(q)\in\so$. In particular, Theorem \ref{thm:key} applies.
\end{lemma}

This lemma allows to find  algebraically simple explicit \bj surfaces with arbitrarily fast rotating normal. We conclude this section with examples.


\subsection{Circular Helicoids}\label{sec:circular}
For our first example, let $q(t)=(\cos(t/2),0,0,-\sin(t/2))$ be  a great circle in $\s^3$. Let $Q(t) = \Phi(q(t))$.

Then 
\[
Q(t)=
\begin{pmatrix}
 \cos (t) & \sin (t) & 0 \\
 -\sin (t) & \cos (t) & 0 \\
 0 & 0 & 1 \\
\end{pmatrix}
\ .
\]

 Integrating the first column gives the core curve
 \[
 c(t) = \begin{pmatrix}
 \sin (t) \\
 \cos (t) \\
 0 \\ \end{pmatrix} \ ,
\]
a circle in the $xy$-plane.
The rotating normal is given as a linear combination of the second and third column as 
\[
n(t) = \cos(at+b)
\begin{pmatrix}
 \sin (t) \\
 \cos (t) \\
 0 \\
\end{pmatrix}
+
\sin(at+b)
\begin{pmatrix}
  0 \\
 0 \\
 1 \\
  \end{pmatrix} \ .
\]
We can assume that $b=0$ as other choices of $b$ will only rotate the surface about the $z$-axis, unless $a=0$, in which case the surface will be  a plane or catenoid.
These are the \bj data of the bent helicoids studied in \cite{mewe1}. Note that when $a\in \Z+\frac12$, the surface is non-orientable. The case of $a=\frac12$ is Meeks' minimal \mo strip (\cite{me}).
See Figure \ref{fig:circle72} when $a=7/2$. In Figure \ref{fig:circle4} the choice is $a=2$.

The \we data of these surfaces are given by
\begin{align*}
G(z) &=(\sin (z)+i \cos (z)) \frac{\cos (a z) }{1-\sin (a z)}\\
dh &= i \cos(az) \, dz \ .
\end{align*}
After the substitution $w=e^{i z}$ this becomes

\begin{align*}
G(w) &=\frac1w  \frac{w^a+i}{w^a-i} \\
dh &= \frac12\frac{w^{2a}+1}{w^{a+1}} \, dw \ .
\end{align*}
This shows that the surface is (for $a$ a positive integer) defined on $\C^*$, is regular, the Gauss map has degree $a+1$, and hence the surface has finite total curvature $-4\pi (a+1)$.


\def\fw{2.7in}
\begin{figure}[H]
	\centering
	\begin{subfigure}[t]{\fw}
  		\centering
		\includegraphics[width=\fw]{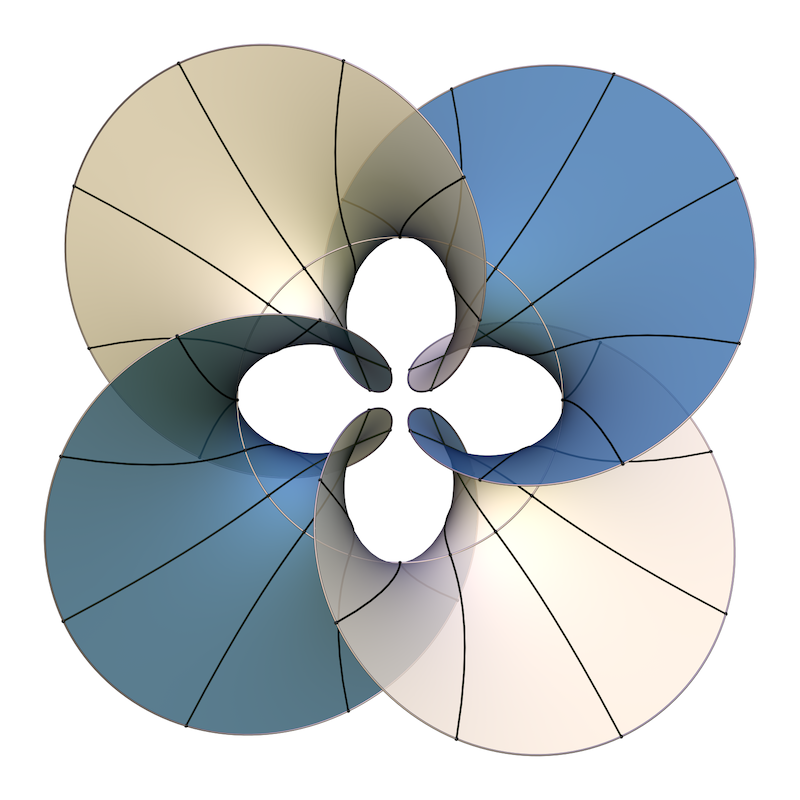}
  		\caption{$a=2$}
		\label{fig:circle4}
	\end{subfigure}
	\quad
	\begin{subfigure}[t]{\fw}
  		\centering
		\includegraphics[width=\fw]{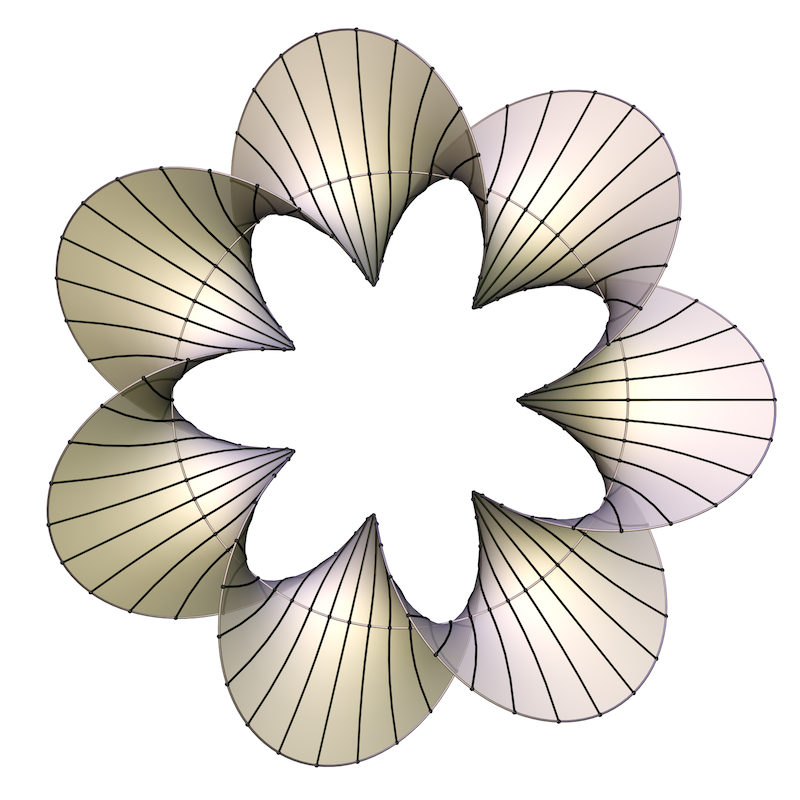}
  		\caption{$a=7/2$}
		\label{fig:circle72}
	\end{subfigure}
 	 \caption{\bj surfaces with circular core.}
 	 \label{fig:circle}
\end{figure}

%
%

\subsection{Torus Knots}

As a second simple example, we apply  the quaternion method to torus knots  in $\R^4$. Let 
\[
q_0(t) = \cos(At){\bf 1}+\cos(Bt){\bf i}+ \sin(B t){\bf j}+ \sin(A t){\bf k} \ ,
\]
and define
\[
q(t) = \frac12 q_0(t)\cdot ({\bf 1}+{\bf i}+{\bf j}+{\bf k})
\]
in order to move $q_0$ away from a standard position and to eventually simplify the \we representation of the minimal surfaces we obtain. Then
\[
Q(t) = \begin{pmatrix}
  \sin (2 B t)-\sin (2 A t) & 2 \sin ((A+B) t) & \cos (2 A t)+\cos (2 B t) \\
 \cos (2 A t)-\cos (2 B t) & -2 \cos ((A+B) t) & \sin (2 A t)+\sin (2 B t) \\
 2 \cos ((A-B) t) & 0 & 2 \sin ((A-B) t) \\
  \end{pmatrix}
 \]
 Integrating the first column gives the space curve
 \[
 c(t) =  \frac{1}{2}\begin{pmatrix}
\frac{\cos (2 A t)}{A}-\frac{\cos (2 B t)}{B} \\
  \frac{\sin (2 A t)}{A}-\frac{\sin (2 B t)}{B} \\
 \frac{4 \sin ((A-B) t)}{A-B} \\
 \end{pmatrix} \ .
\]
Observe that while the curve we begin with is a torus knot, the resulting space curve has no reason to be knotted.
The rotating normal is given as the normalized linear combination of the second and third column as 
\[
n(t) =  \cos(at+b)
\begin{pmatrix}
  \sin ((A+B) t) \\
 -\cos ((A+B) t) \\
 0 \\
  \end{pmatrix}
+\frac12
\sin(at+b)
\begin{pmatrix}
 \cos (2 A t)+\cos (2 B t) \\
 \sin (2 A t)+\sin (2 B t) \\
 2 \sin ((A-B) t) \\
  \end{pmatrix} \ .
\]

In the simplest case for $a=0$ and $b=0$ the resulting \bj surfaces are generalized Enneper surfaces. For instance, for $A=-\frac12$ and $B=\frac32$, we obtain the standard Enneper surface as shown in Figure \ref{fig:enneper0}. The ``hole'' in the center will eventually close. If we rotate the normal by $90^\circ$ by letting $b=\pi/2$, we obtain the surface in Figure \ref{fig:enneper1} with two ends.

\def\fw{2.7in}
\begin{figure}[H]
	\centering
	\begin{subfigure}[t]{\fw}
  		\centering
		\includegraphics[width=\fw]{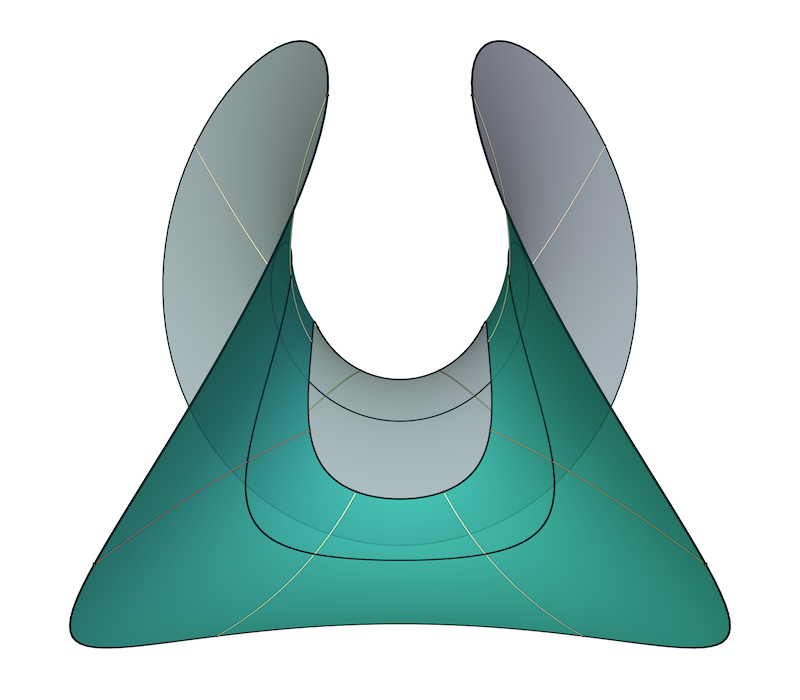}
  		\caption{$b=0$}
		\label{fig:enneper0}
	\end{subfigure}
	\quad
	\begin{subfigure}[t]{\fw}
  		\centering
		\includegraphics[width=\fw]{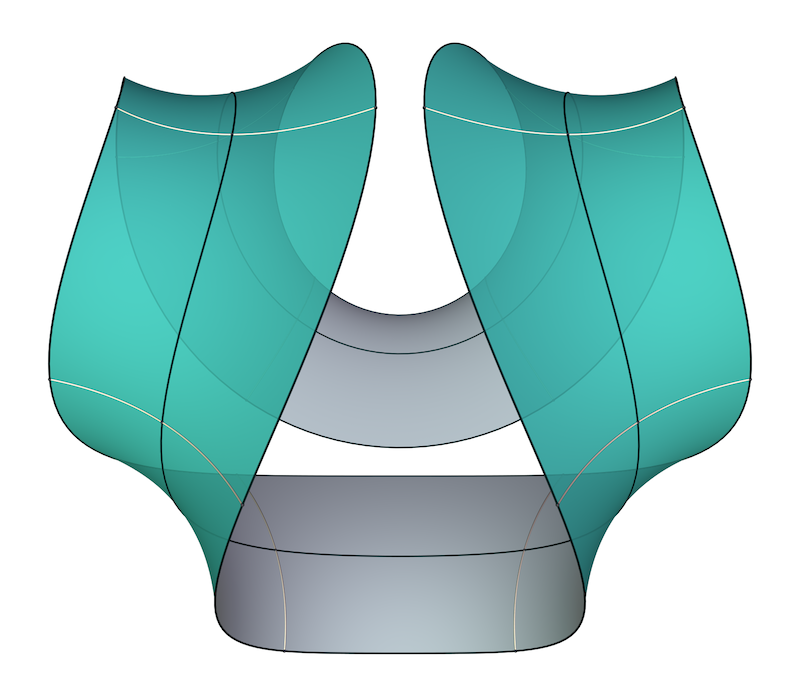}
  		\caption{$b=\pi/2$}
		\label{fig:enneper1}
	\end{subfigure}
 	 \caption{\bj surfaces related to Enneper's surface}
 	 \label{fig:enneper}
\end{figure}

Increasing $a$ creates helicoidal surfaces along the core curve, as in Figure \ref{fig:enneper20} for $a=20$.
Enneper surfaces with $k$-fold dihedral symmetry can be obtained by using $A=-\frac12$ and $B=\frac12(2k-1)$. The planar Enneper surfaces with $k$-fold dihedral symmetry (\cite{ka5})
arise if we choose  $A=+1/2$ and  $B=\frac12(2k+1)$. Examples with 3-fold dihedral symmetry and no twist are shown in Figure \ref{fig:planarenneper}, and a  version with the same core curve but faster rotating normal appears in Figure \ref{fig:PlanarEnneper50}.

\def\fw{2.7in}
\begin{figure}[H]
	\centering
	\begin{subfigure}[t]{\fw}
  		\centering
		\includegraphics[width=\fw]{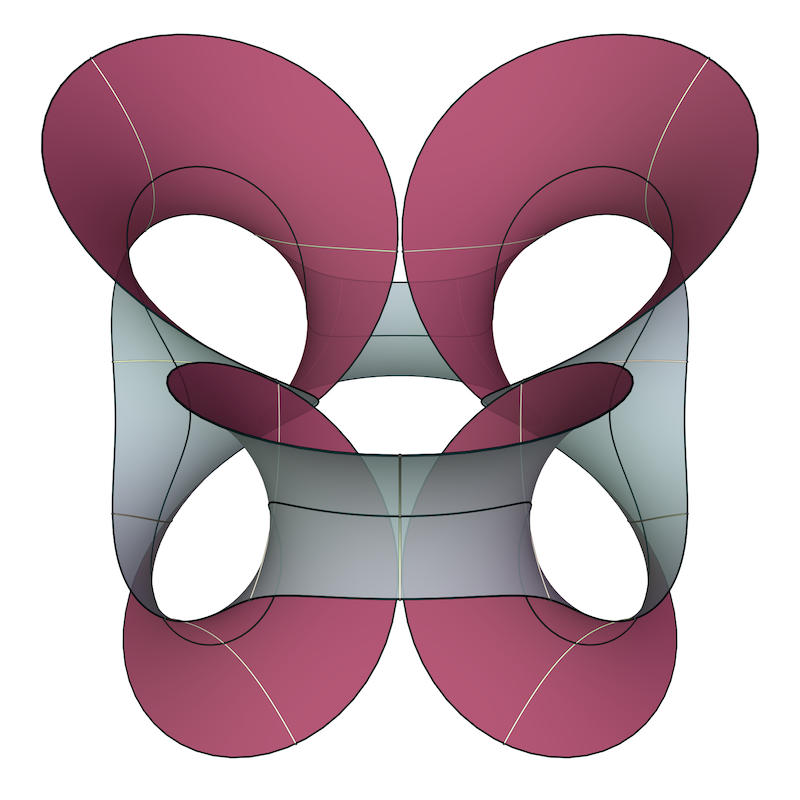}
  		\caption{$b=0$}
		\label{fig:planarenneper0}
	\end{subfigure}
	\quad
	\begin{subfigure}[t]{\fw}
  		\centering
		\includegraphics[width=\fw]{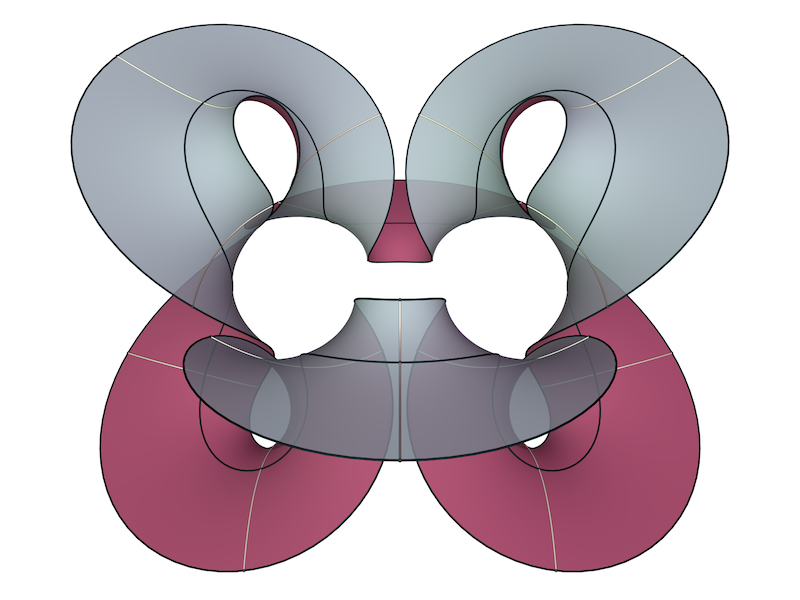}
  		\caption{$b=\pi/2$}
		\label{fig:planarenneper1}
	\end{subfigure}
 	 \caption{\bj surfaces related to planar Enneper's surface}
 	 \label{fig:planarenneper}
\end{figure}

For other (rational) choices of $A$ and $B$, the surfaces will be immersed with two ends, regular, and of finite total curvature. To see this, we compute from the unintegrated \bj formula the Gauss map and height differential (using Equation \ref{eqn:wei}) in the coordinate $w$ given by $z=e^{i w}$.

\begin{align*}
G(w) &= -\frac{i w^{A+B} \left(e^{i b} w^{a+A}+e^{i b} w^{a+B}+w^A-w^B\right)}{e^{i b} w^{a+A}-e^{i b}
   w^{a+B}+w^A+w^B}\\
dh &=\frac{i}{2}  \frac{w^{2 B} \left(e^{i b} w^a-1\right)^2-w^{2 A} \left(e^{i
   b} w^a+1\right)^2}
   {e^{i b} w^{a+A+B+1} }  \, dw \ .
\end{align*}
If $a=b=0$ this simplifies to 
\begin{align*}
G(w) &= -i w^{A+B}\\
dh &=-2 i w^{A-B-1}\, dw \ .
\end{align*}
which are the \we data of the generalized Enneper surfaces (see \cite{ka5}), as claimed.

For $a>0$ a positive integer, we limit the regularity discussion to the standard Enneper case when $A=-\frac12$,  $B=\frac32$ and $b=0$ in order to keep the formulas simple. Let

\begin{align*}
P(w) &=i w \left(w^{a+2}+w^a-w^2+1\right)\\
Q(w)&= w^{a+2}-w^a-w^2-1  \ . \\
\end{align*}

Then 
\begin{align*}
G(w) &= \frac{P(w)}{Q(w)}\\
dh &=\frac1{2w^{a+4}}P(w)Q(w)  \, dw \ .
\end{align*}

To see that the surface is regular in all of $\C^*$, we need to show  that the conformal factor does not vanish and has no singularities. This is equivalent to  $P$ and $Q$ having no common roots.  Suppose that $w$ is a common root of $P$ and $Q$. Then also
$0=P(w)-i w Q(w) = 2i w(w^a+1)$. Thus $w=0$ or $w^a+1=0$. But $Q(0)=1$ and if $w^a=-1$, then $Q(w) = -2w^2$. This implies that in $\C^*$ the height differential has a zero if and only if the Gauss map has a zero or pole of matching order. This in turn implies that the surface is regular.

We also see that the Gauss map has degree $a+3$ when $a>0$. This degree drops to 1 if $a=0$ due to cancellations.

\def\fw{2.7in}
\begin{figure}[H]
	\centering
	\begin{subfigure}[t]{\fw}
  		\centering
		\includegraphics[width=\fw]{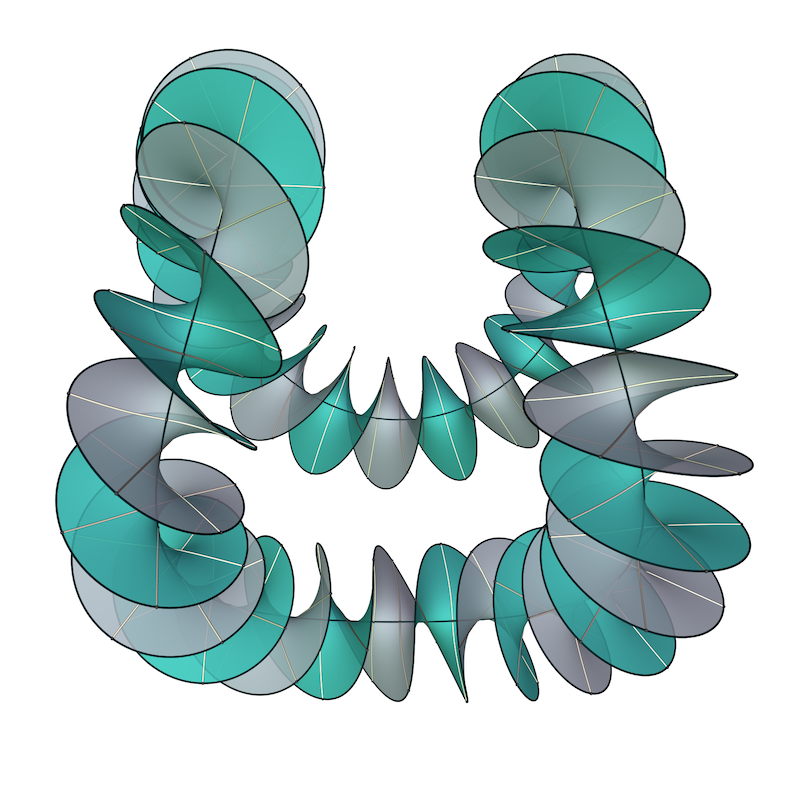}
  		\caption{$A=-\frac12$, $B=\frac32$ and $a=20$}
		\label{fig:enneper20}
	\end{subfigure}
	\quad
	\begin{subfigure}[t]{\fw}
  		\centering
		\includegraphics[width=\fw]{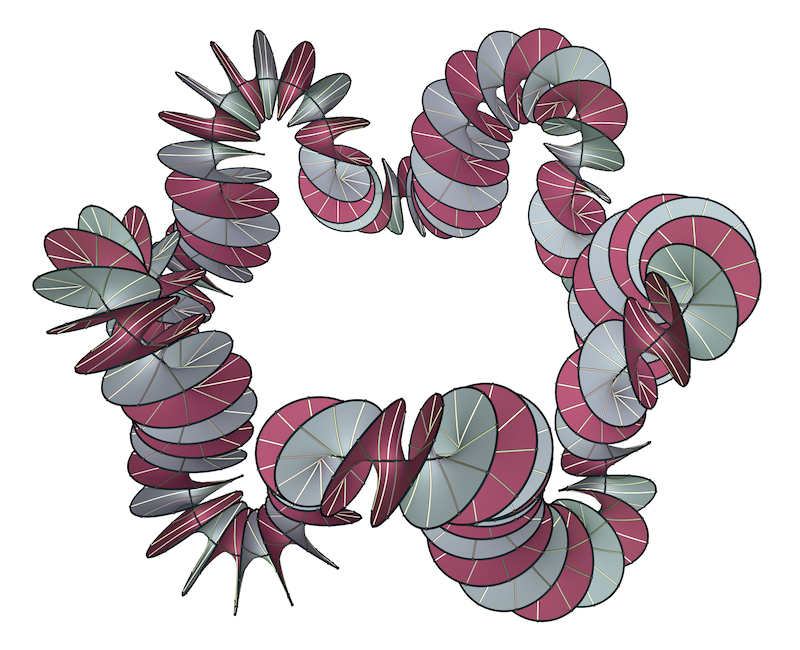}
  		\caption{$A=+\frac12$, $B=\frac72$ and $a=50$}
		\label{fig:PlanarEnneper50}
	\end{subfigure}
 	 \caption{\bj surfaces based on Enneper core curves}
 	 \label{fig:twistenneper}
\end{figure}

\subsection{Periodic Surfaces}\label{sec:snake}

So far, the core curves of the examples we have considered have been closed curves. This is in general not the case. As an example, we consider the entry curve $q(t)$ given as the quaternion product of two great circles of $\s^3$. Let 
\begin{align*}
q_1(t) &=\cos(t/2){\bf j} + \sin(t/2){\bf k}  \\
q_2(t) &=-\cos(t/2){\bf 1} + \sin(t/2) {\bf k} 
\end{align*}
Then

\begin{align*}
q(t) &= q_1(t)\cdot q_2(t) \\
&= -\sin(t/2)^2 \, {\bf 1} +\frac12 \sin(t)\,  {\bf i} - \cos(t/2)^2\,  {\bf j} -\frac12 \sin(t) \, {\bf k}
\end{align*}

and 

\[
Q(t)=
\begin{pmatrix}
 -\cos (t) & -\sin (t) & 0 \\
 -\cos (t) \sin (t) & \cos ^2(t) & \sin (t) \\
 -\sin ^2(t) & \cos (t) \sin (t) & -\cos (t) \\
\end{pmatrix}
\]
so that the  core curve becomes $c(t)=\left(-\sin (t),\frac12 \cos ^2(t),\frac{1}{4} \sin (2 t)-\frac{t}{2}\right)$ and the rotating normal is given by
\[
n(t) =  \cos(at+b)
\begin{pmatrix}
  -\sin ( t) \\
 \cos^2 ( t) \\
 \cos(t)\sin(t) \\
  \end{pmatrix}
+
\sin(at+b)
\begin{pmatrix}
 0 \\
\sin(t) \\
-\cos(t) \\
\end{pmatrix} \ .
\]
Note that this curve is periodic in the $z$-direction and projects onto the $xy$ plane as a singular piece of the parabola $y=2(1-x^2)$. We will come back to this example from a different point of view in Section \ref{sec:lissajous}.

Using the coordinate $w$  on $\C^*$ with $z=e^{i w}$ we obtain as the \we representation of the surface divided by its translational symmetry
\begin{align*}
G(w) &= \frac{P(w)}{Q(w)}\\
dh &=\frac{e^{-i b}}{8w^{a+3}}  \, dw
\end{align*}
with
\begin{align*}
P(w) &= e^{i b} \left(w^2+1\right) w^a-i (w+2) w+i\\
Q(w) &= e^{i b} ((w-2) w-1) w^a-i \left(w^2+1\right) 
\end{align*}

In general, the degree of the Gauss map is $a+2$ except when $a=0$ and $b=\pi/2$ (see Figure \ref{fig:TransEnneper}), when the degree is 1. In this case, the \we representation simplifies to
\begin{align*}
G(w) &= \frac{w-1}{w+1}\\
dh &=\frac{i}{2w^3} (w^2-1)  \, dw
\end{align*}
with  an annular Scherk end at 0 and a higher order end at $\infty$.
In any case, the surfaces are regular everywhere in $\C^*$. This can for instance be seen by computing the resultant of $P(w)$ and $Q(w)$ as
\[
\Res(P,Q) = -8^{a+1}i^a e^{(a+2) i b}
\]
which never vanishes.

\def\fw{2.8in}
\begin{figure}[H]
	\centering
	\begin{subfigure}[t]{\fw}
  		\centering
		\includegraphics[width=\fw]{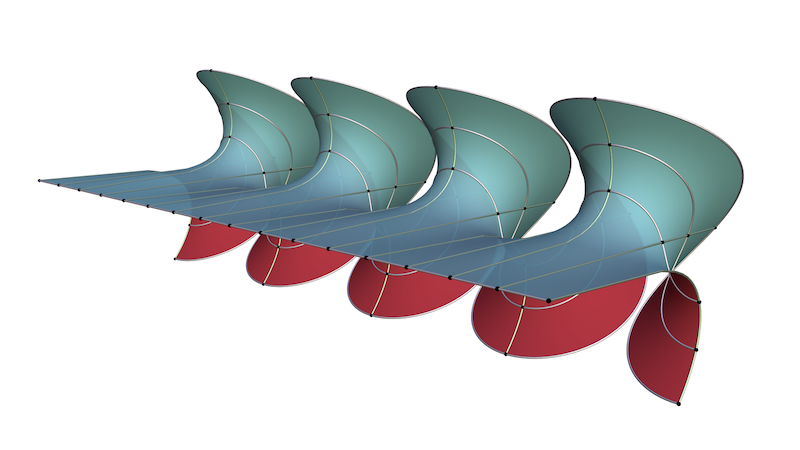}
  		\caption{$a=0$, $b=\frac\pi2$}
		\label{fig:TransEnneper}
	\end{subfigure}
	\begin{subfigure}[t]{\fw}
  		\centering
		\includegraphics[width=\fw]{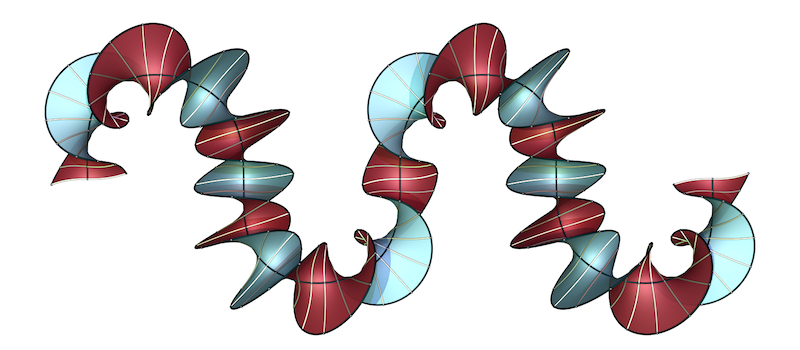}
  		\caption{$a=5$, $b=0$}
		\label{fig:snake}
	\end{subfigure}
 	 \caption{Periodic \bj surfaces based on the quaternion product of two great circles}
 	 \label{fig:periodic}
\end{figure}

%
%
%
%
%
%

\section{Lifting Plane Curves}

The major drawback of the the quaternion method is that it gives little control over the geometry of the core curve. In the translation invariant examples that we created with the quaternion method we noticed that the core curve often had a simple projection onto the plane perpendicular to the translation. This suggested the question whether one could prescribe a planar curve $(x(t), y(t))$ and lift it to a space curve $c(t)=(x(t), y(t),z(t))$
such that $c'(t)$ is the first column of a matrix $Q(t)=\Phi(q(t))$ for a suitable curve $q(t)\in \R^4$. Of course all this should happen in the realm of polyexp functions. 

To our delight, this is indeed possible. Moreover, the matrices $\Phi(q(t))$ we obtained this way turned out to be rather special elements of $\R\cdot \so$, namely $180^\circ$ rotations followed by scalings. While this does not achieve full generality, it allows for a very simple description and a highly effective method.

More precisely, we have:

\begin{theorem}\label{thm:lift}
Given a planar polyexp curve $(x(t),y(t))$,  there is family of  explicit polyexp curves 
\[
\Psi_\lambda(t) \in \R\cdot  \so
\]
depending on a parameter $\lambda \in\R$,
such that the projection of the integral 
$c(t)$ of the first column of $\Psi_\lambda(t) $ onto the $xy$-plane is the curve $(x(t),y(t))$.
\end{theorem}

\begin{proof}

Recall that for a  (column) vector $v\in \R^n$, the symmetric matrix 
\[
R(v) = 2v\cdot v^t - v^t v I_n
\]
is the $180^\circ$ degree rotation about the line in the direction of $v$, followed by a scaling by $|v|^2$. Explicitly, in $\R^3$, for $v=(x,y,z)$, 
\[
R(x,y,z) = \Phi(x{\bf i} + y {\bf j} + z {\bf k}) = \begin{pmatrix}
 x^2-y^2-z^2 & 2 x y & 2 x z \\
 2 x y & -x^2+y^2-z^2 & 2 y z \\
 2 x z & 2 y z & -x^2-y^2+z^2 \\
\end{pmatrix}
\]

In particular, $R(x,y,z) \in \R \so$, and when $x$, $y$, and $z$ are polyexp in $t$, then so is  $R(x,y,z)$. Note that this matrix is also symmetric, which implies that it represents a $180^\circ$ rotation followed by a scaling.

Now let a polyexp planar curve $(x(t), y(t))$ be given, and fix  a constant $\lambda\in \R$. Define the polyexp curve
\[
\Psi_\lambda(t) = \frac1{2\lambda} R(x'(t), y'(t),\lambda) \ .
\] 
Then $\Psi_\lambda(t)$ satisfies the assumptions of Theorem \ref{thm:key} with 
\[
\mu(t) = \frac1{2\lambda} \left( \lambda^2+x'(t)^2+y'(t)^2\right) \ ,
\]
and hence its columns can be used to find an explicit \bj surface. 
We intentionally choose the {\em third} column of $\Psi$ as $c'$, namely
\[
c'(t) = 
\begin{pmatrix}
x'(t) \\
y'(t) \\
 \frac1{2 \lambda }\left( \lambda ^2-x'(t)^2-y'(t)^2\right)\\
\end{pmatrix} \ .
\]
This is dictated by the desire to lift a curve in the $xy$-plane and to use symmetric matrices.
With the appropriate integration constants,  the space curve $c$ then projects onto the $xy$-plane as the given curve $(x(t), y(t))$. 
\end{proof}

\begin{remark}
We note that for a {\em closed} planar curve $(x(t), y(t))$ defined on an interval $[t_0, t_1]$, the  constructed lifts will in general not be closed but  periodic with a translational period in the $z$-direction given by
\[
T = \frac12\lambda (t_1-t_0) -  \frac1{2\lambda}\int_{t_0}^{t_1}  x'(t)^2+y'(t)^2  \, dt \ .
\]
This shows, however, that for a suitable choice of  $\lambda$, we can {\em always} obtain closed lifts.
\end{remark}

\begin{remark}\label{rem:other}
One can carry out this construction also for non-constant $\lambda$ as long as the function $ \left( x'(t)^2+y'(t)^2\right)/\lambda(t)$ is polyexp. We will see an example in Section \ref{sec:lissajous}.
\end{remark}

In order to define a rotating normal along $c$, let
\begin{align*}
n_1(t) &= \frac1{2\lambda} \begin{pmatrix}
-\lambda^2+x'(t)^2-y'(t)^2\\
2{x'(t) y'(t)}\\
 2\lambda x'(t) \\
\end{pmatrix}\\
n_2(t) &= \frac1{2\lambda} \
\begin{pmatrix}
2x'(t) y'(t)\\
-\lambda ^2-x'(t)^2+y'(t)^2 \\
2\lambda  y'(t) \\\end{pmatrix}
\end{align*}
be  the first two columns of $\Psi_\lambda(t)$. Then define for real parameters $a$ and $b$ the normal
\[
n(t) =\frac{2\lambda}{\lambda^2+x'(t)^2+y'(t)^2}\left(
\cos(at + b)
n_1(t)
+\sin(at+b)
n_2(t)
\right) \ .
\]

We  use the pair $\{c,n\}$ as \bj data. As both $c(t)$ and $n(t)$ are polyexp, the corresponding \bj surface will be explicit.

We will show next that the surfaces constructed this way are almost always regular. 

\begin{theorem}
Let $(x(t), y(t))$ be a polyexp plane curve,  let $c(t)$ be the polyexp curve in $\R^3$ constructed in Theorem \ref{thm:lift}, and $n(t)$ 
the normal defined above. Assume that $\cos(at+b) x'(t) + \sin(at+b) y'(t)$ is {\em not} identical equal to 0; this will be true for all but at most one choice of real numbers $(a,b)$.
Then the \bj surface given by these data is defined in the entire complex plane and  regular for a generic choice of $\lambda$.
\end{theorem}

\begin{proof}
We introduce the functions
\begin{align*} 
P(w) &=\frac{i e^{-\frac{1}{2} i (a w+b)}}{\sqrt{2} \sqrt{\lambda }}
 \left(x'(w)+i  y'(w) +\lambda  e^{i (a w+b)}\right)	\\
Q(w) &= \frac{i e^{-\frac{1}{2} i (a w+b)}}{\sqrt{2} \sqrt{\lambda }}
\left(e^{i (a w+b)} \left(x'(w)-i y'(w)\right)-\lambda\right) \ ,
\end{align*}
in which we will express the \we data of the \bj surface.
The unintegrated \bj formula gives us the \we representation
\begin{align*}
c'(w) - i\cdot  c'(w) \wedge n(w) &=
\begin{pmatrix}
 x'(w) \\
 y'(w) \\
\frac1{2 \lambda } \left(\lambda ^2-x'(w)^2-y'(w)^2\right)\\
  \end{pmatrix}+\\
  &\qquad
+  i\cos(aw+b)
  \begin{pmatrix}
 \frac1{\lambda }{x'(w) y'(w)} \\
 -\frac1{2 \lambda }\left(\lambda ^2+x'(w)^2-y'(w)^2\right) \\
 y'(w) \\
 \end{pmatrix}
 \\
 &\qquad
  +i \sin(aw+b)
  \begin{pmatrix}
  \frac1{2 \lambda } \left(\lambda ^2-x'(w)^2+y'(w)^2\right)\\
 -\frac1{\lambda } {x'(w) y'(w)}\\
 -x'(w) \\
   \end{pmatrix}\\
\end{align*}

Solving  Equation (\ref{eqn:wei})  for the Weierstrass data $G$ and $dh$  yields (after a tedious computation) 

\begin{align*} 
G(w) &= \frac{P(w)}{Q(w)}	\\
dh &=  P(w)Q(w) \, dw \ .\\
\end{align*}

Note that $P$ and $Q$ are entire functions. If they do not vanish   simultaneously at a point $w$, then whenever $dh$ vanishes at $w$, $G$ must have a zero or pole of the same order at $w$, which implies that the surface is regular  at $w$. So we need to show that for a generic choice of $\lambda$, $P$ and $Q$ do never vanish simultaneously.

Solving both equations $P(w)=0$ and $Q(w)=0$ for $\lambda$, we obtain
\[
\lambda = -e^{-i (a w+b)} \left(x'(w)+i y'(w)\right) = e^{i (a w+b)} \left(x'(w)-i y'(w)\right)
\]
By the identity theorem, the set of points $w$ where the second of these two equations is satisfied will either be a discrete subset of the complex plane, or the entire complex plane. 
In the first case, we just avoid the discrete set of values where the two expressions agree. In the second case, we note that the second equation is equivalent to
\[
\cos(aw+b) x'(w) + \sin(aw+b) y'(w) = 0 \ ,
\]
which must now hold for all $w$, violating our assumption.
\end{proof}

\begin{remark}\label{rem:singular}
We briefly discuss the condition on $a$ and $b$. In the case that $\cos(aw+b) x'(w) + \sin(aw+b) y'(w) = 0$ holds for all $w$, we necessarily have
\[
\begin{pmatrix}
x'(w)\\ y'(w)
\end{pmatrix} = r(w) \cdot
\begin{pmatrix}
-\sin(aw+b) \\ \cos(a w+b)
\end{pmatrix}
\]
for a polyexp function $r(w)$. In this case, we obtain $\lambda = -i \cdot r(w)$. This means that if we choose $w$ so that $\lambda = -i \cdot r(w)$ is real, the polynomials $P$ and $Q$ will have a common root at $w$ for the choice of $a$, $b$, and $\lambda$, and hence the minimal surface will be singular at $w$.  In other words, for this choice of $a$ and $b$, for no choice of $\lambda$ the surface will be regular in the entire complex plane. In section \ref{sec:cycloids}, we will give an example for this behavior.

Similarly, in section \ref{sec:ellipse} we will give an example that where  isolated choices of $\lambda$  lead to  surfaces with singularities.
\end{remark}

\begin{remark}
The formula for the Gauss map $G$ in the proof can also be used to determine the total curvature in case the plane curve is trigonometric.
\end{remark}

\section{Examples}

In this section, we will apply Theorem \ref{thm:lift} to  some classical planar curves. Except for the first example, all the surfaces we obtain are new.

\subsection{Circles}\label{sec:circles}

The lifts of circles, parametrized by arc length, will either be circles or helices.

Let $x(t)=\cos(t)$ and $y(t)=\sin(t)$. Then, 
\begin{align*}
\Psi_\lambda(t) &= \frac1{2\lambda} R(x'(t), y'(t),\lambda) \\
&=\frac1{2\lambda}
\begin{pmatrix}
 -\lambda ^2-\cos(2t) & -\sin (2t) & -2 \lambda  \sin (t) \\
 - \sin (2t) & -\lambda ^2+\cos (2t) & 2 \lambda  \cos (t) \\
 -2 \lambda  \sin (t) & 2 \lambda  \cos (t) & \lambda ^2-1 \\ 
  \end{pmatrix} \ .
\end{align*}

The integral of the third column  gives the core curve

\[
\begin{pmatrix}
x(t)\\
y(t)\\
 z(t)\\
\end{pmatrix}
=
\begin{pmatrix}
 \cos (t) \\
 \sin (t) \\
 \frac{t \left(\lambda ^2-1\right)}{2 \lambda } \\
 \end{pmatrix}
\]
which is a circle (if $|\lambda|=1$) or a helix. We have discussed the circular case in Section \ref{sec:circular}. Images of \bj surfaces based on a helix are in Figure \ref{fig:helices}.

Using the first two columns, we can form and simplify the rotating normal
\[
n(t) =\frac1{1+\lambda^2}
\begin{pmatrix}
-\cos ((a-2) t+b)  -\cos (a t+b) \lambda ^2 \\
 \sin ((a-2) t+b)- \sin (a t+b)\lambda ^2 \\
 2  \sin ((a-1) t+b)  \lambda\\
 \end{pmatrix} \  .
\]
The \bj integral can easily and explicitly be evaluated, but the equations are not illuminating. More interesting are the \we data, which can be written after the substitution $w = -i \log(z)$ as

\begin{align*}
G(w) &= \frac{P(w)}{Q(w)} \\
dh & = -\frac{ie^{-i b}}{2\lambda w^{a+1}} P(w) Q(w)\, dz
\end{align*}

with

\begin{align*}
P(w) &= w-i e^{i b} \lambda  w^a\\
Q(w) &=i \lambda -e^{i b} w^{a-1} \ .
\end{align*}
This shows that if $a$ is a positive integer,  the \we data of the surface are defined on the punctured plane $\C^*$, 
and the Gauss map has degree $a$. Furthermore, $P$ and $Q$ don't have a root in common, because otherwise $P-i\lambda w Q = w(\lambda^2+1)$ had a root at the same point $w$. This would mean that $w=0$, but $Q(0)\ne0$.
This implies that  the surface is regular everywhere in $\C^*$.

In case when $a=0$ or $a=1$, the degree of the Gauss map is 1, and the surface is in the family of associated surfaces of the catenoid.

\def\fw{2.7in}
\begin{figure}[H]
	\centering
	\begin{subfigure}[t]{\fw}
  		\centering
		\includegraphics[width=\fw]{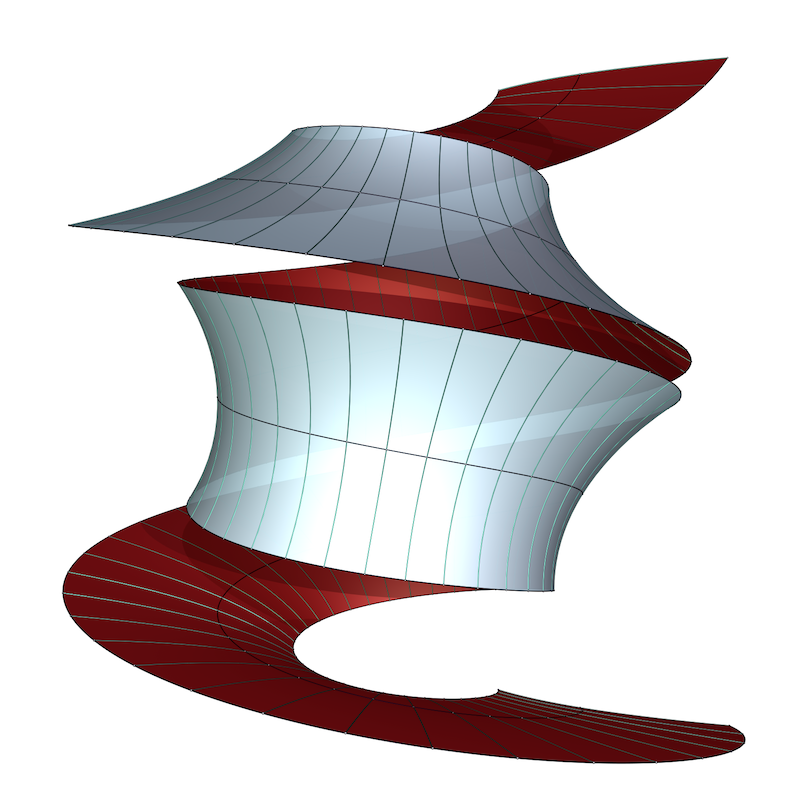}
  		\caption{$a=0$}
	\end{subfigure}
	\quad
	\begin{subfigure}[t]{\fw}
  		\centering
		\includegraphics[width=\fw]{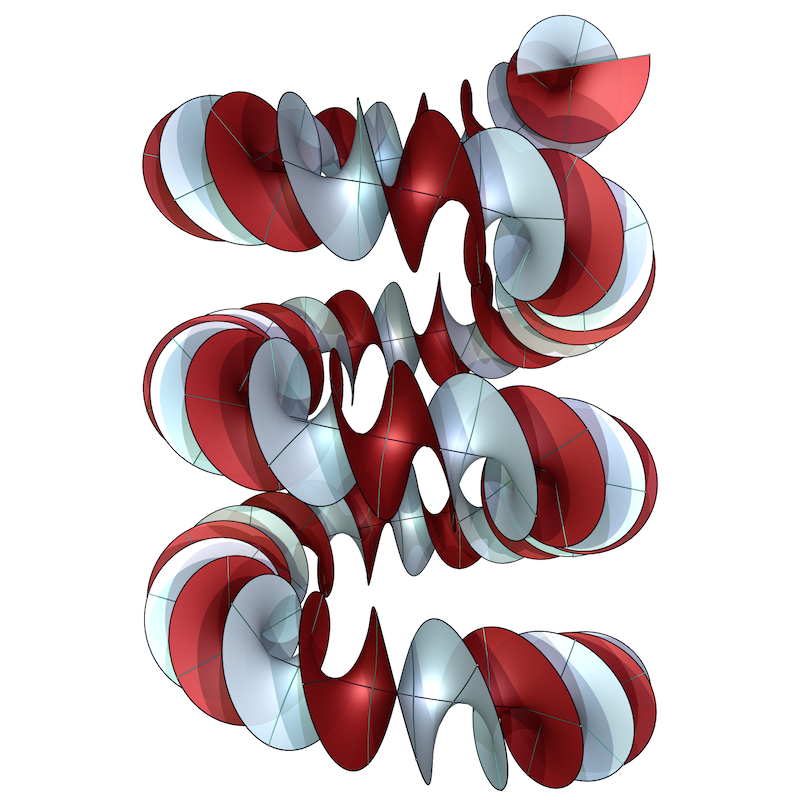}
  		\caption{$a=10$}
	\end{subfigure}
 	 \caption{\bj surfaces based on a helix}
 	 \label{fig:helices}
\end{figure}

These \bj surfaces based on helices can also be easily obtained using the quaternion method,
starting with  a (non-great) circle of   $\s^3$.

With a free parameter $\sigma\in (0,\pi/2)$, let
\[
q(t) =  \sin(\sigma) {\bf 1}+ \cos(\sigma)\cos(t){\bf j}+ \cos(\sigma)\sin(t){\bf k} \ .
\]
 We obtain
 \[
Q(t)=
\begin{pmatrix}
 -\cos (2 \sigma ) & - \sin (t) \sin (2\sigma ) & \cos (t) \sin (2 \sigma ) \\
 \sin (t) \sin (2 \sigma ) & \cos (2 t) \cos ^2(\sigma )+\sin ^2(\sigma ) & \cos ^2(\sigma ) \sin (2 t) \\
 - \cos (t) \sin (2\sigma )& \cos ^2(\sigma ) \sin (2 t) & \sin ^2(\sigma )-\cos (2 t) \cos
   ^2(\sigma )
\end{pmatrix} \ ,
\]
and the core curve becomes the (horizontal) helix
\[
c(t)= -\left(  \cos (2 \sigma ) t , \sin (2 \sigma ) \cos (t), \sin (2\sigma )  \sin (t) \right) \ .
\]

Using this curve together with the normal vector (assuming $b=0$ without loss of generality thanks to the screw motion invariance of helices)

\[
n(t) = \cos(at) 
\begin{pmatrix}
 -\sin (t) \sin (2\sigma ) \\
 \cos (2 t) \cos ^2(\sigma )+\sin ^2(\sigma ) \\
 \cos ^2(\sigma ) \sin (2 t) \\
 \end{pmatrix}
 +\sin(at)
 \begin{pmatrix}
 \cos (t) \sin (2 \sigma ) \\
 \cos ^2(\sigma ) \sin (2 t) \\
 \sin ^2(\sigma )-\cos (2 t) \cos ^2(\sigma ) \\
  \end{pmatrix}
 \]
we obtain the same \bj surfaces as above, rotated by $90^\circ$.

\subsection{Ellipses}\label{sec:ellipse}

Creating an explicit \bj surface with normal rotating along a planar ellipse  leads to elliptical integrals. Our lifting method avoids this problem by slightly bending the ellipse into a spatial curve. 
We illustrate this with the ellipse $x(t)=\cos( t)$ and $y(t)=3\sin( t)$. This will also serve as an example where a particular choice of $\lambda$ can lead to a minimal surface with singularities.

We compute the third coordinate as

\[
z(t) = \frac1{2\lambda} \left( (\lambda^2-5) t - 2\sin(2t) \right) \ .
\]
This curve closes when $\lambda=\sqrt{5}$. We use the curve $c(t)=(x(t), y(t), z(t))$ as core curve and chose  the  rotating normal with $a=2$ and $b=0$ so that it is given by
\[
n(t) = \frac1{\lambda ^2+4 \cos (2 t)+5}
\begin{pmatrix}
 -\left(\lambda ^2+4\right) \cos (2 t)-\cos (4 t)-4 \\
 \left(4-\lambda ^2\right) \sin (2 t)+\sin (4 t) \\
 2 \lambda  (2 \sin (t)+\sin (3 t)) \\ 
\end{pmatrix} \ .
\]
After changing the coordinate $z$ to $w$ using $w=e^{i z}$, the \we data of the resulting \bj surface are given by

\begin{align*}
G(w) &= \frac{-1+w^2 (-2+i \lambda  w)}{w \left(-i \lambda +w^3+2 w\right)} \\
dh & = \frac{\left(-i \lambda +w^3+2 w\right) \left(w^2 (\lambda  w+2 i)+i\right)}{2 \lambda 
   w^4} \, dw
\end{align*}

This shows that the surface is defined and regular in $\C^*$ {\em unless} the numerator and denominator of $G$ have a common root. This happens when $|\lambda|=1$. In case $\lambda=1$, the common roots are at 
$w=\frac{i}{2} \left(\sqrt{5}-1\right)$ and $w=-\frac{i}{2}  \left(\sqrt{5}+1\right)$. In Figure \ref{fig:ellipse} we show the regular surface for $\lambda=\sqrt5$ (when the core curve is closed) on the left, and the singular periodic surface with $\lambda =1$ on the right.

\def\fw{2.8in}
\begin{figure}[H]
	\centering
	\begin{subfigure}[t]{\fw}
  		\centering
		\includegraphics[width=\fw]{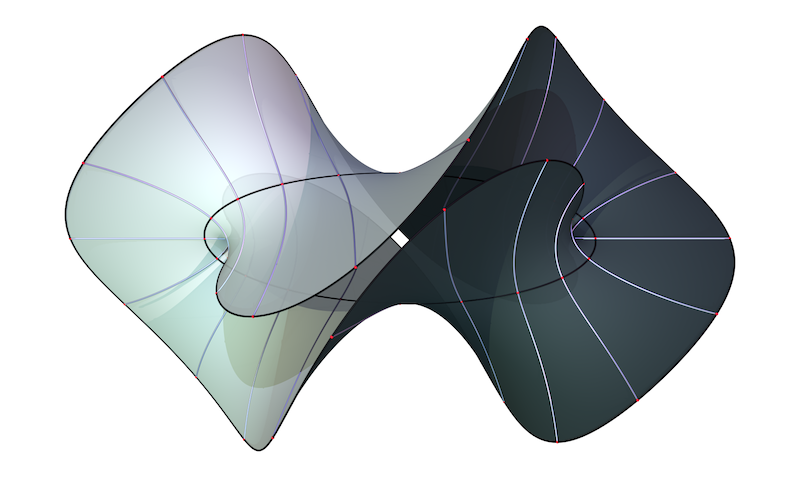}
  		\caption{$\lambda=\sqrt5$}
	\end{subfigure}
	\begin{subfigure}[t]{\fw}
  		\centering
		\includegraphics[width=\fw]{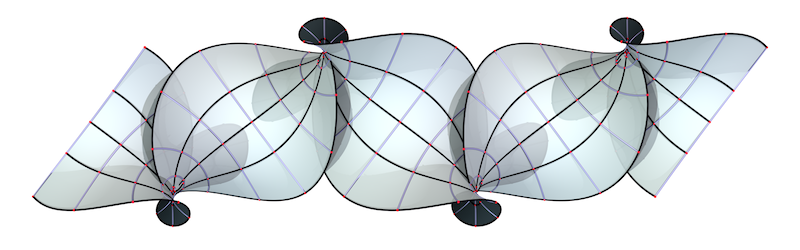}
  		\caption{$\lambda=1$}
	\end{subfigure}
 	 \caption{\bj surfaces based on an ellipse}
 	 \label{fig:ellipse}
\end{figure}

\subsection{Lissajous Curves}\label{sec:lissajous}

\def\fw{1.5in}
\begin{figure}[H]
	\centering
	\begin{subfigure}[t]{\fw}
  		\centering
		\includegraphics[width=\fw]{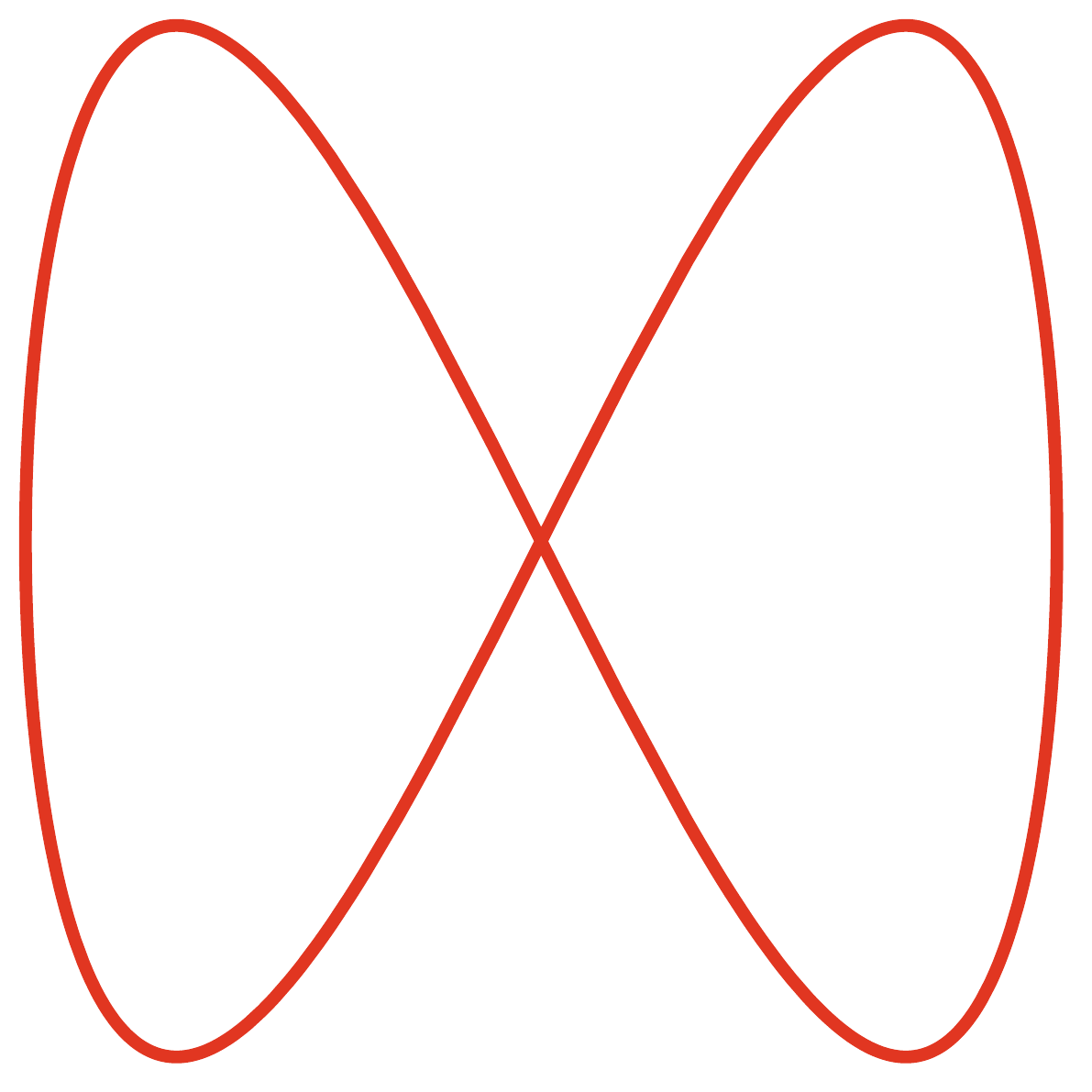}
  		\caption{$(1,2)$}
	\end{subfigure}
	\quad
	\begin{subfigure}[t]{\fw}
  		\centering
		\includegraphics[width=\fw]{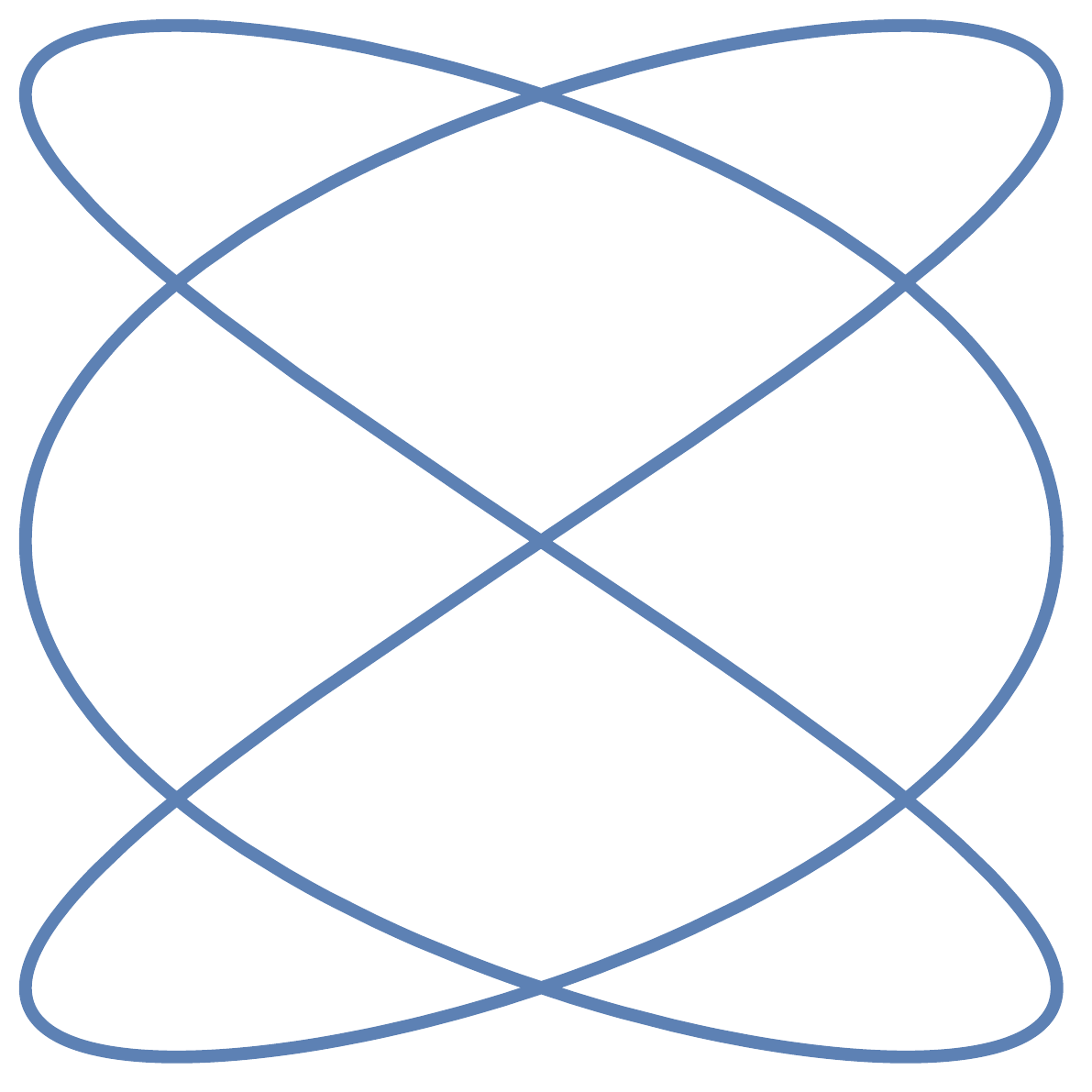}
  		\caption{$(3,2)$}
	\end{subfigure}
 	 \caption{Two Lissajous curves}
 	 \label{fig:Lissajous}
\end{figure}

As another class of trigonometric curves, we can consider the $(\xi,\eta)$-Lissajous curves given by
$x(t)=\cos(\xi t)$ and $y(t)=\sin(\eta t)$ with integer parameters $\xi$ and $\eta$. Their lifts have as $z$ coordinate
\[
z(t) =\frac1{8\lambda}
\left(-2 \left(\xi ^2+\eta ^2-2 \lambda ^2\right)t +\xi  \sin (2 \xi  t)-\eta  \sin
   (2 \beta  t)
   \right) \ .
\]
We can, as observed before, make this a closed curve by choosing $\lambda$ such that
\[
2\lambda^2=\xi^2+\eta^2 \ .
\]

By changing the value of $b$ in the rotating normal
\[
n(t) =\frac{2\lambda}{\lambda^2+x'(t)^2+y'(t)^2}\left(
\cos(at + b)
n_1(t)
+\sin(at+b)
n_2(t)
\right) \ ,
\]
one can uniformly rotate the normal about the core curve. For lines, circles, or helices as core curves, the effect is just a translation, rotation, or screw motion of the surface, but for other curves, the appearance can change significantly.

In Figure \ref{fig:Lissajous21} we show the \bj surfaces for the lift of the $(1,2)$-Lissajous curves with $\lambda=2$, $a=1$, and two different values of $b$. 

\def\fw{2.7in}
\begin{figure}[H]
	\centering
	\begin{subfigure}[t]{\fw}
  		\centering
		\includegraphics[width=\fw]{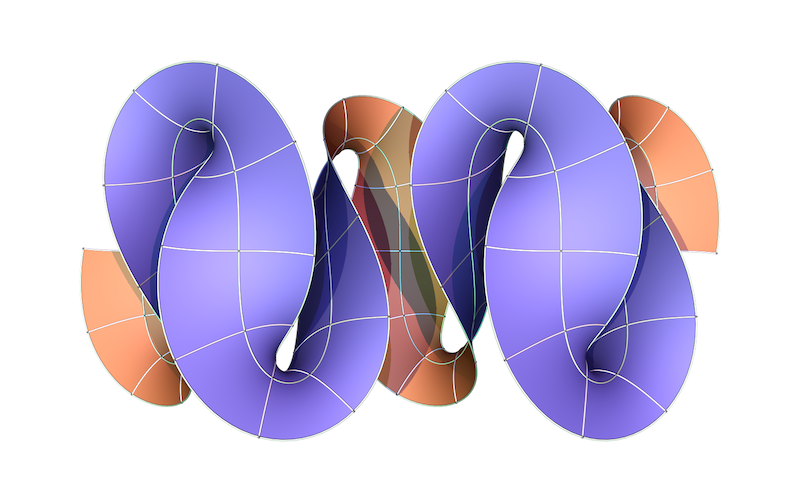}
  		\caption{$b=0$}
	\end{subfigure}
	\quad
	\begin{subfigure}[t]{\fw}
  		\centering
		\includegraphics[width=\fw]{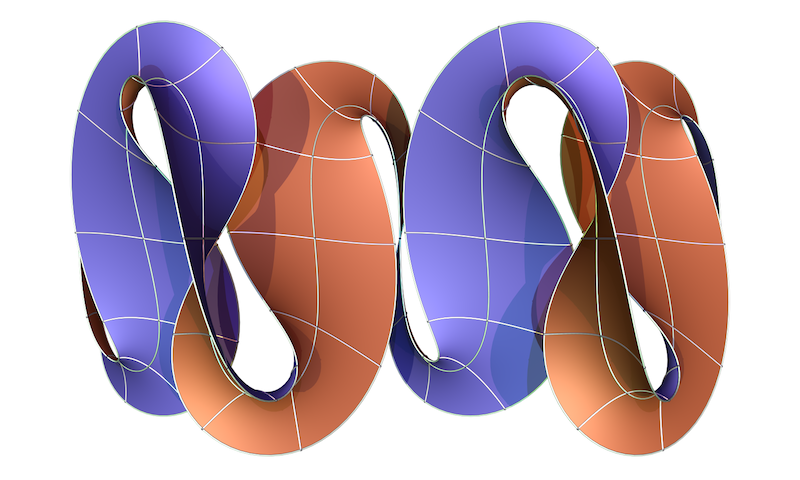}
  		\caption{$b=\pi/2$}
	\end{subfigure}
 	 \caption{\bj surfaces for the $(1,2)$-Lissajous curve with $a=1$.}
 	 \label{fig:Lissajous21}
\end{figure}

In Section \ref{sec:snake} we used the quaternion method to construct a \bj surface with core curve
\[
c(t)=\left(-\sin (t),\frac{\cos ^2(t)}{2},\frac{1}{4} \sin (2 t)-\frac{t}{2}\right) \ .
\]
This core curve projects onto the $xy$-plane as the curve 
\[
(x(t), y(t)) = \left(-\sin (t),\frac{\cos ^2(t)}{2}\right) = \left(-\sin (t),\frac{1}{4} (\cos (2 t)+1)\right) \ ,
\]
which essentially is a Lissajous curve. Applying the lifting method to this curve gives the $z$ coordinate as
\[
z(t) = \frac1{64\lambda}
\left(4 t \left(8 \lambda^2 -5\right)-8 \sin (2 t)+\sin (4 t)\right) \ ,
\]
which is significantly more complicated than what we obtained with the quaternion method. Following Remark \ref{rem:other}, we notice that we may chose as $\lambda = \lambda(t)$ any factor of 
\[
x'(t)^2+y'(t)^2 = \left(\sin ^2(t)+1\right) \cos ^2(t) \ .
\]
and will still obtain integrable \bj data. In fact, using $\lambda(t) = -(\sin(t)^2+1)$ produces exactly the same core curve as in Section \ref{sec:snake}. The quaternion method is still more general, because the curves in $\R\cdot \so$ produced by the lifting method are always multiples of $180^\circ$ rotations.

\subsection{Cycloids}\label{sec:cycloids}

\def\fw{1.3in}
\begin{figure}[h]
	\centering
	\begin{subfigure}[t]{\fw}
  		\centering
		\includegraphics[width=\fw]{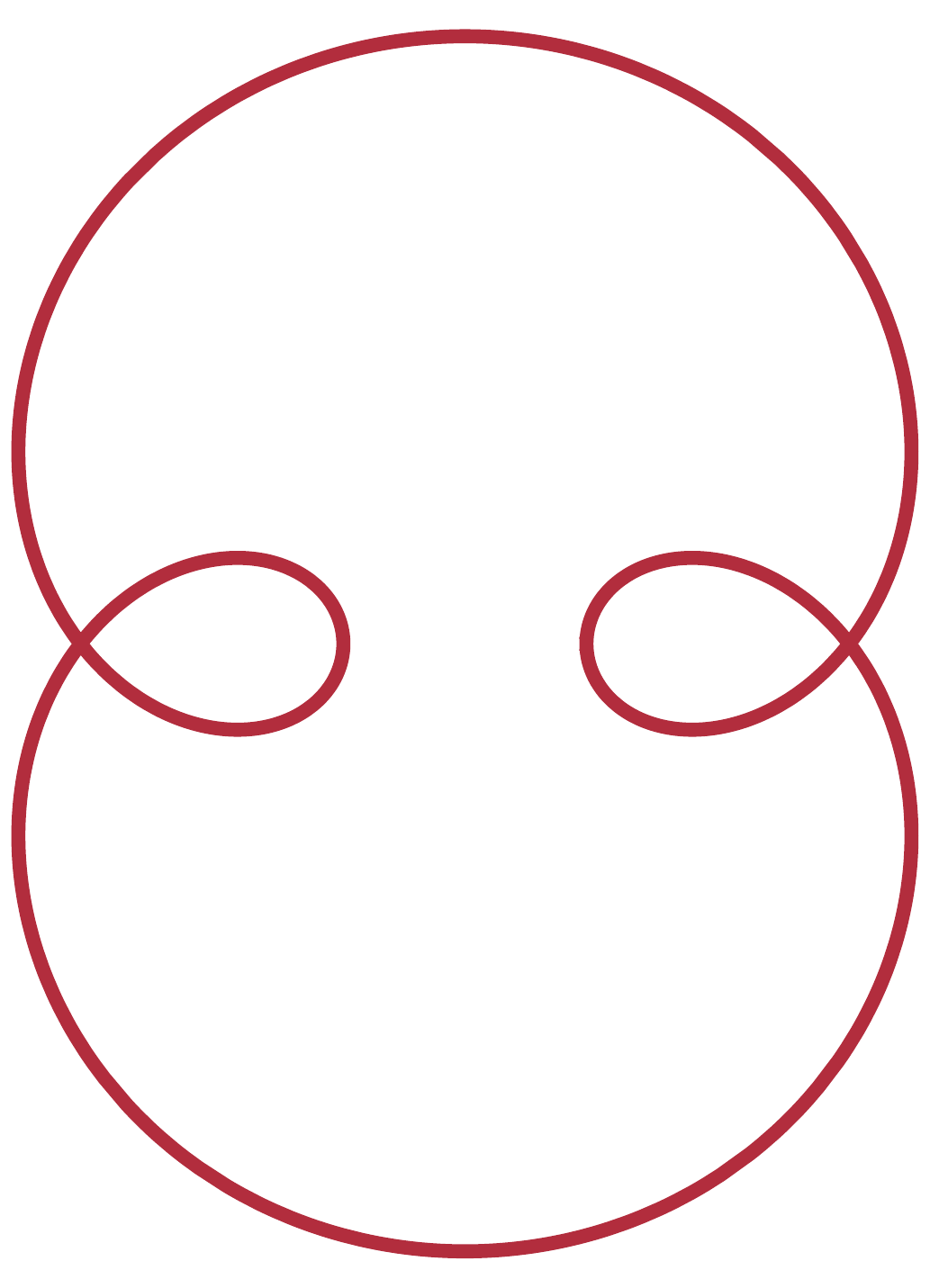}
  		\caption{order 2 cycloid}
	\end{subfigure}
	\quad
	\begin{subfigure}[t]{\fw}
  		\centering
		\includegraphics[width=\fw]{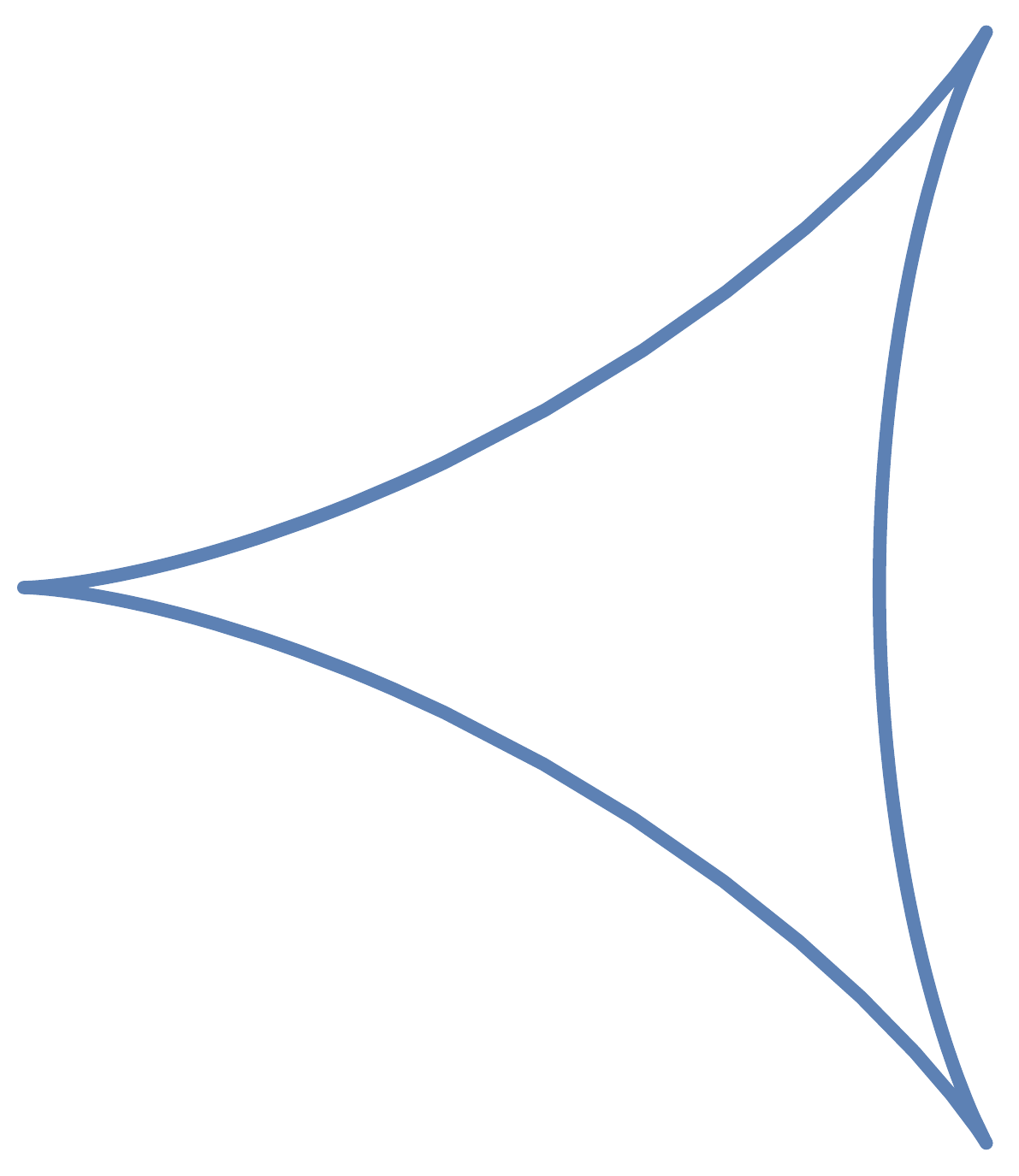}
  		\caption{Deltoid}
	\end{subfigure}
 	 \caption{Two cycloids}
 	 \label{fig:cycloids}
  \end{figure}

The general cycloid is the trace of a chosen point in the plane of  a circle that rolls along another fixed circle. If the fixed circle has radius $R$, the rolling circle radius $r$, and the tracing point in the plane of the rolling circle has distance $s$ from the center of the rolling circle, the cycloid  can be parametrized as
\[
\begin{pmatrix}
x(t)\\
y(t)
\end{pmatrix} = (R+r) \begin{pmatrix}
\cos(t)\\
\sin(t)
\end{pmatrix} 
-s \begin{pmatrix}
\cos((1+R/r)t)\\
\sin((1+R/r)t)
\end{pmatrix} \ .
\]

\def\fw{2.8in}
\begin{figure}[h]
	\centering
	\begin{subfigure}[t]{\fw}
  		\centering
		\includegraphics[width=\fw]{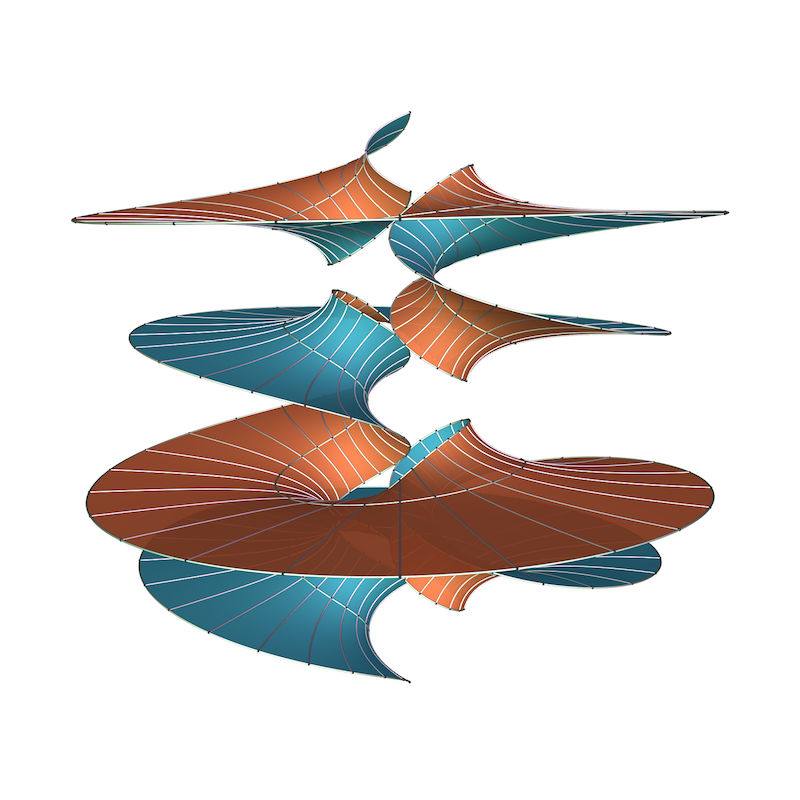}
  		\caption{$a=2$ and $\lambda=6$}
		\label{fig:cyc2}
	\end{subfigure}
	\begin{subfigure}[t]{\fw}
  		\centering
		\includegraphics[width=\fw]{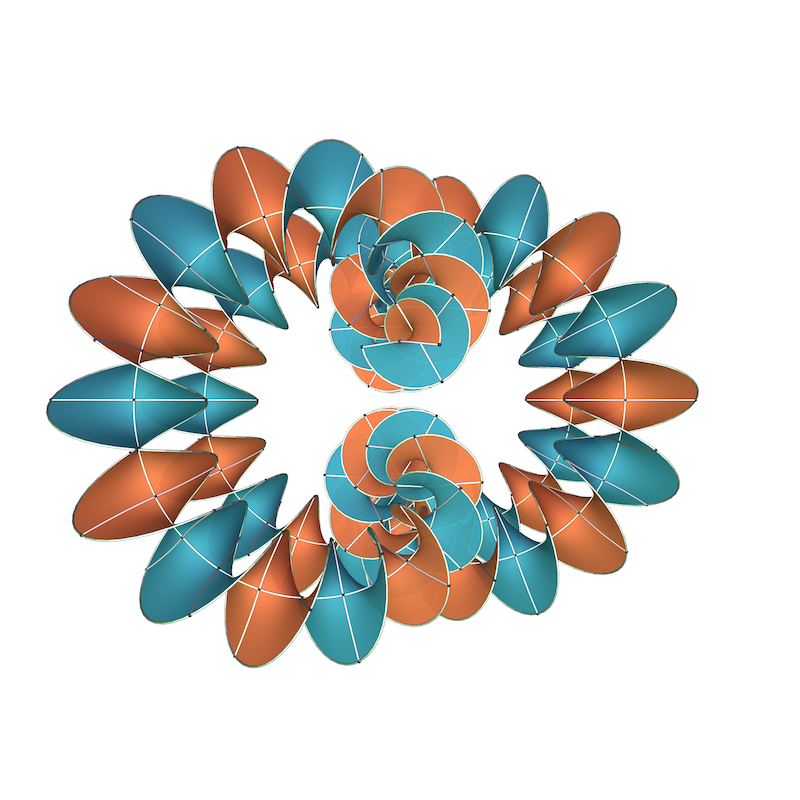}
  		\caption{$a=20$, closed curve}
		\label{fig:cyc20}
	\end{subfigure}
 	 \caption{\bj surfaces based on a lifted cycloid}
 	 \label{fig:cycloid2min}
  \end{figure}
  
We will discuss two special cases. 
Let 
\[
(x(t),y(t)) = (3 \cos (t)-2 \cos (3 t),3 \sin (t)-2 \sin (3 t)) \ .
\]
The $z$-coordinate of  the lifted curve becomes
\[
z(t) = 
 \frac1{2 \lambda }  \left(t \left(\lambda ^2-45\right)+18 \sin (2 t)\right)
   \]
so that the lifted curve  is closed when $\lambda = 3\sqrt{5}$. Observe that the self intersections disappear in the lift. This is a common but not universal phenomenon --- for instance, the closed lifts of Lissajous curves still have self intersections. We show the resulting \bj surface in the closed case with $a=20$ in Figure \ref{fig:cyc20},
and in Figure \ref{fig:cyc2} a periodic surface obtained  with $\lambda=6$ and $a=2$. We have chosen $b$  such that large parts of the surface are embedded.

Another simple example of a cycloid is the deltoid, given by

 \[
(x(t),y(t)) = (-2 \cos (t)-\cos (2 t),\sin (2 t)-2 \sin (t)) \ .
\]
The $z$-coordinate of the lift becomes
\[
z(t)=\frac1{6\lambda}\left(3 \left(\lambda ^2-8\right) t+8 \sin (3 t)\right) \ .
\]
Observe that the singularity of the plane deltoid disappears in the lifted curve. For $\lambda=2\sqrt2$ the lift is closed.

\def\fw{2.7in}
\begin{figure}[H]
	\centering
	\begin{subfigure}[t]{\fw}
  		\centering
		\includegraphics[width=\fw]{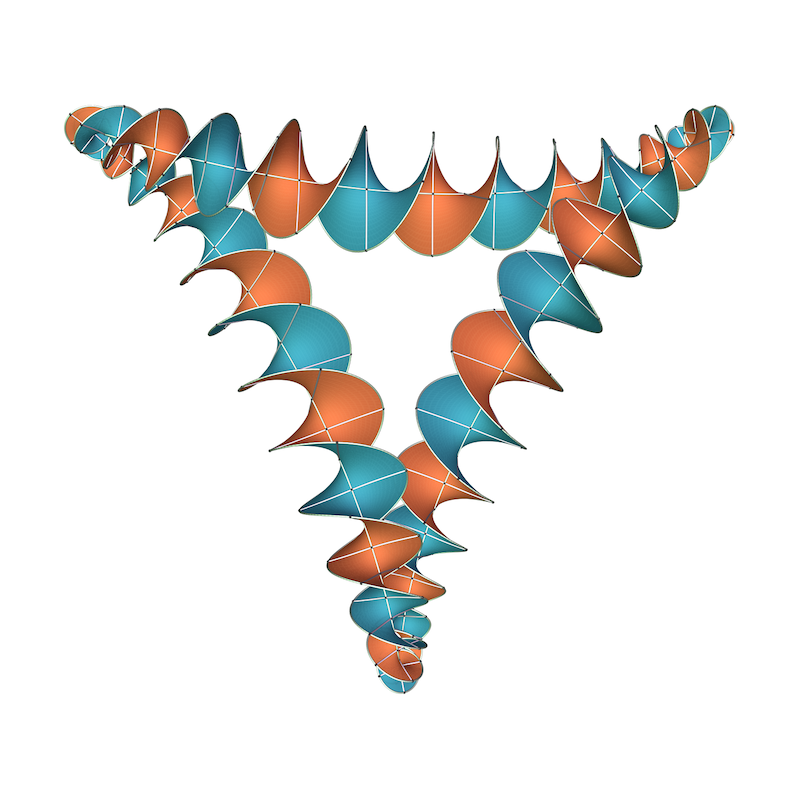}
  		\caption{$a=20, b=0$}
	\end{subfigure}
	\quad
	\begin{subfigure}[t]{\fw}
  		\centering
		\includegraphics[width=\fw]{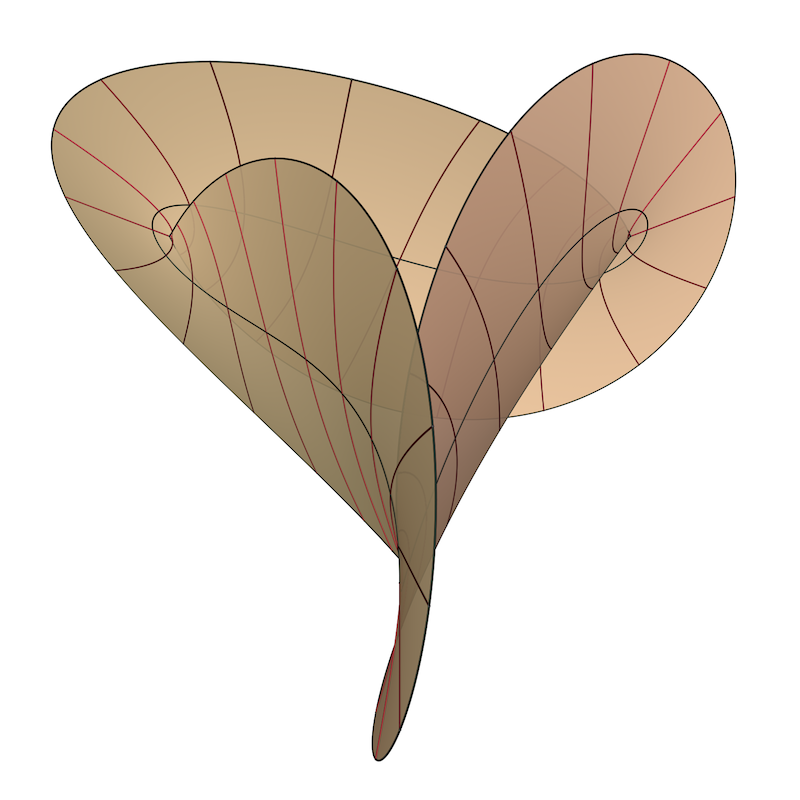}
  		\caption{$a=-\frac12, b=\frac\pi2$}
		\label{fig:singdeltoid}
	\end{subfigure}
 	 \caption{\bj surfaces based on the closed lift of the deltoid}
 	 \label{fig:deltoid}
  \end{figure}
  
 For $a=-\frac12$ and $b=\frac \pi2$, the \bj surface based on the deltoid exhibits the  behavior explained in Remark \ref{rem:singular}. To see this, we note that the tangent vector 
 of the deltoid can be written as

 \[
\begin{pmatrix}
x'(w)\\ y'(w)
\end{pmatrix} = r(w) \cdot
\begin{pmatrix}
-\sin(aw+b) \\ \cos(a w+b)
\end{pmatrix}
\]
with $a=-\frac12$, $b=\frac \pi2$ and $r(w) = -4 \sin(3w/2)$
so that $w$ will be a singularity if we choose $\lambda =  4 i\cdot \sin(3w/2)$. This become apparent in the \we data, given by

\begin{align*} 
G(w) &= -e^{-i w /2}	\\
dh &=  \frac1{2\lambda} \left(\lambda -4 i\cdot \sin(3w/2) \right)^2 \, dw \ .
\end{align*}
We show this singular \bj surface  for $\lambda=2\sqrt2$ when the lifted curve closes in Figure \ref{fig:singdeltoid}. Incidentally, this surface is also a \mo strip.

\subsection{Trefoil Curves}

\def\fw{2in}
\begin{figure}[H]
	\centering
	\begin{subfigure}[t]{\fw}
  		\centering
		\includegraphics[width=\fw]{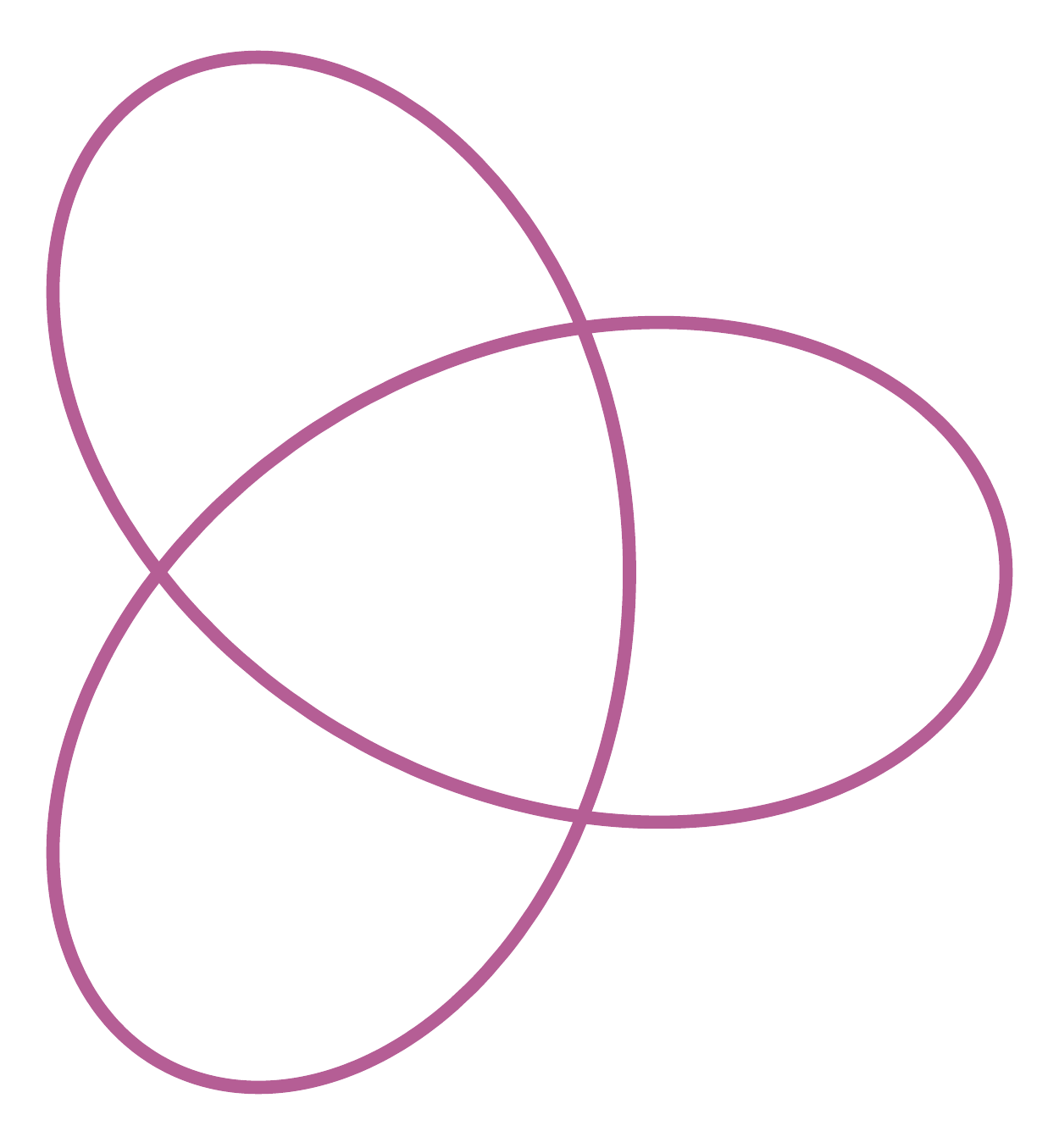}
  		\caption{$\xi=\frac14$}
	\end{subfigure}
	\quad
	\begin{subfigure}[t]{\fw}
  		\centering
		\includegraphics[width=\fw]{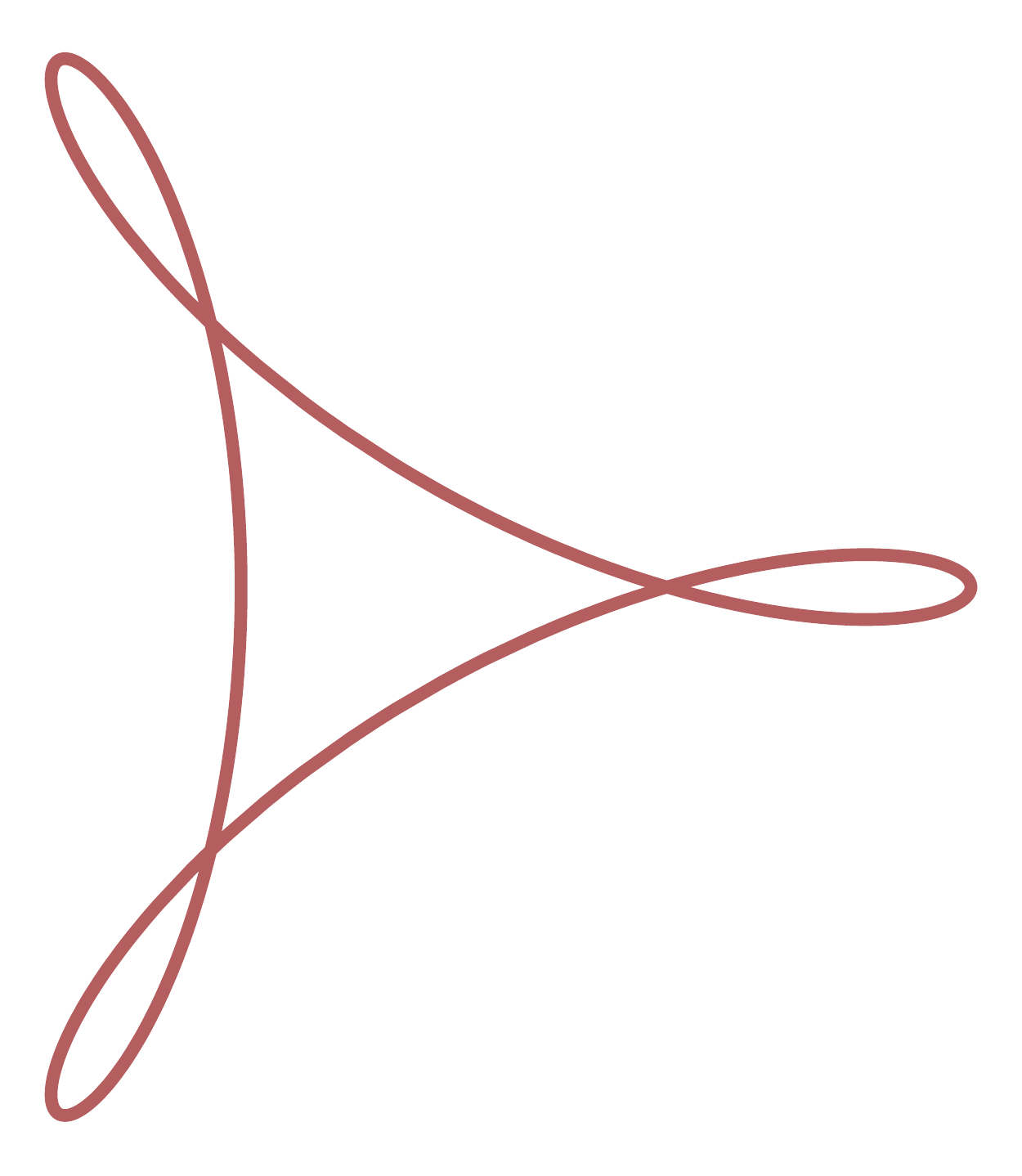}
  		\caption{$\xi=\frac34$}
	\end{subfigure}
 	 \caption{Trefoil Curves}
 	 \label{fig:trefoilcurves}
  \end{figure}

In this section, we construct a knotted \bj surface. The basis of this construction is the following family of curves which we call {\em trefoil curves},

 \[
(x(t),y(t)) =( (\cos(t)-\xi)\cos (t), (\cos(t)+\xi)\sin (t)) \ ,
\]
for a real parameter $\xi$.

The $z$-coordinate of the lift becomes
\[
z(t) = \frac1{6\lambda}\left( 3t(\xi^2+1-\lambda^2)+2\xi \sin(3t)\right) \ ,
\]
which is closed for $\lambda = \sqrt{\xi^2+1}$. Moreover, the lift is knotted for $0<\xi<\frac12$. Choosing $a=\frac12$ and $b=\frac\pi2$ results in  a \bj surface that is an almost horizontal  knotted minimal \mo strip shown in Figure \ref{fig:trefoil2}.

\def\fw{2.7in}
\begin{figure}[H]
	\centering
	\begin{subfigure}[t]{\fw}
  		\centering
		\includegraphics[width=\fw]{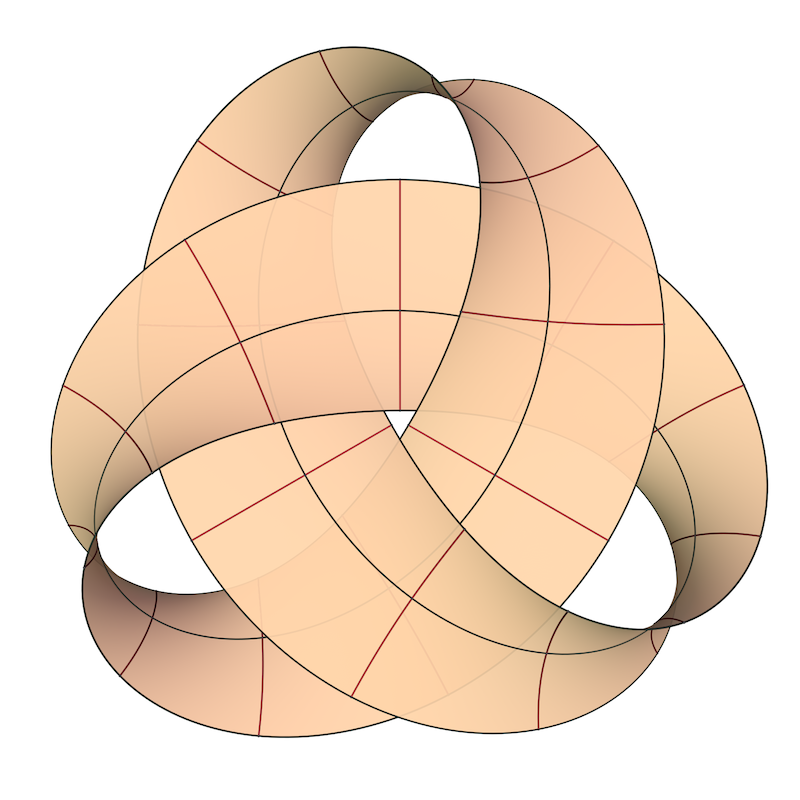}
  		\caption{$a=\frac12, b=\frac\pi2$}
		\label{fig:trefoil2}
	\end{subfigure}
	\quad
	\begin{subfigure}[t]{\fw}
  		\centering
		\includegraphics[width=\fw]{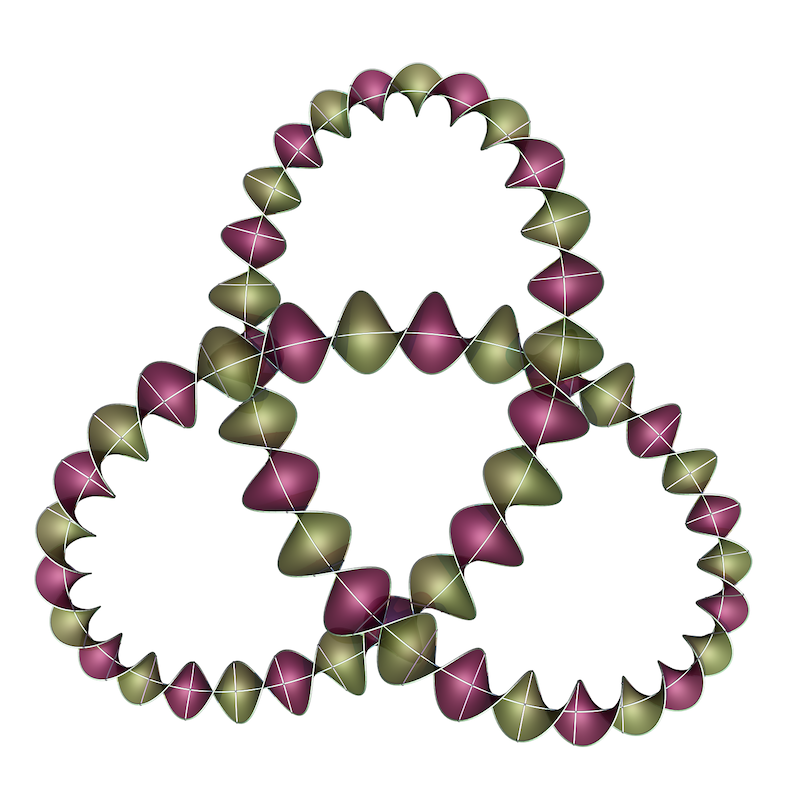}
  		\caption{$a=30, b=0$}
	\end{subfigure}
 	 \caption{\bj surfaces based on the trefoil curve with $\xi=\frac14$}
 	 \label{fig:trefoil30}
  \end{figure}

The \we data of the oriented cover of the \bj surface for $a=\frac12$ and $b=\frac\pi2$ are given in the coordinate $w=e^{i z/2}$ by

\begin{align*}
G(w) &=\frac{i w \left(\xi +w^6+\lambda  w^3\right)}{\xi  w^6-\lambda  w^3+1} \\
dh & =\frac{i \left(\xi +w^6+\lambda  w^3\right) \left(\xi  w^6-\lambda  w^3+1\right)}{ \lambda  w^7} \, dw \ .
\end{align*}
One can show that numerator and denominator of $G(w)$ have no common roots in $\C^*$, which implies that the non-oriented surface is complete, regular, and of finite total curvature $-14\pi$.

\subsection{Spirals}

\def\fw{2in}
\begin{figure}[H]
	\centering
	\begin{subfigure}[t]{\fw}
  		\centering
		\includegraphics[width=\fw]{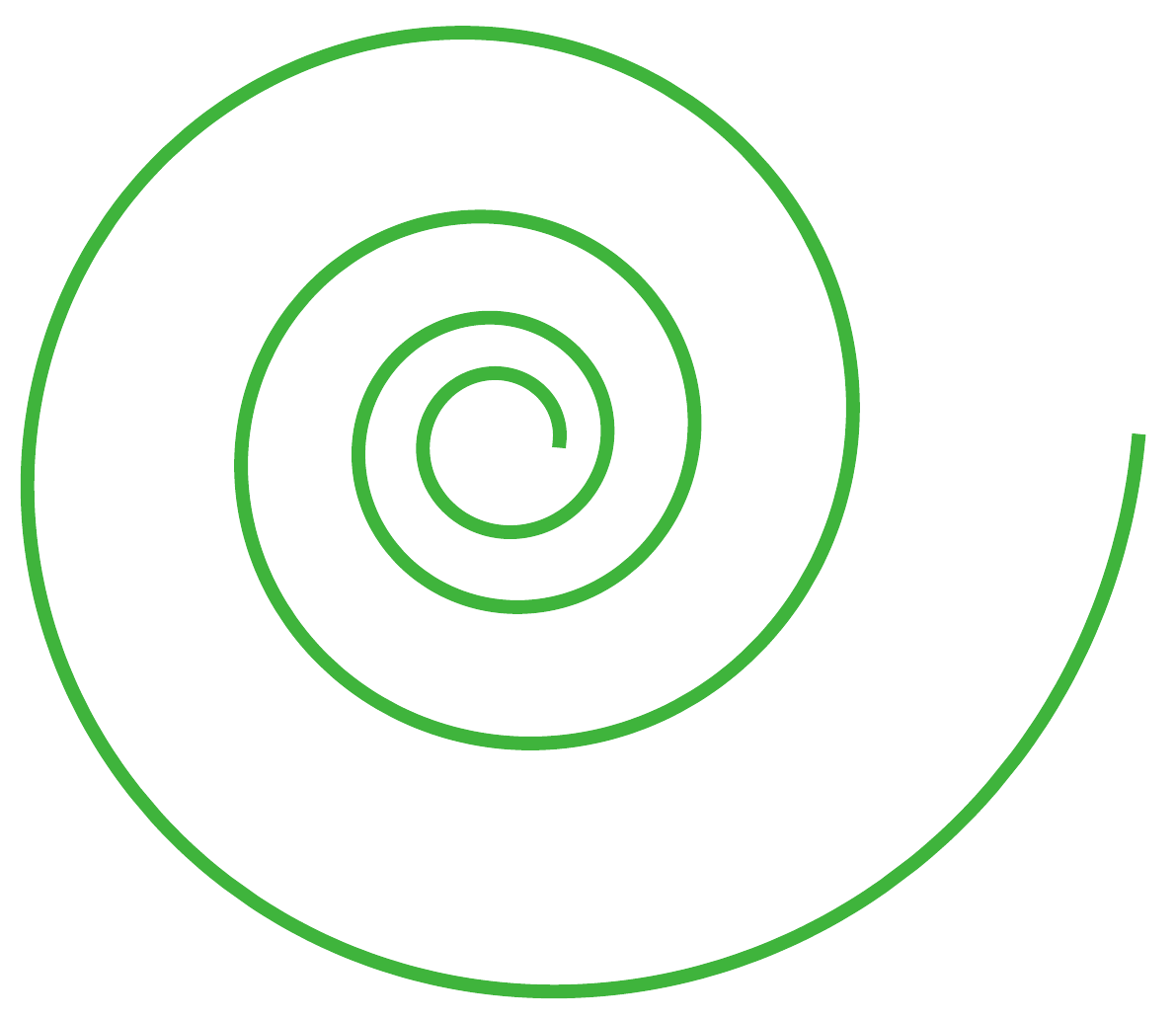}
  		\caption{Logarithmic Spiral}
	\end{subfigure}
	\quad
	\begin{subfigure}[t]{\fw}
  		\centering
		\includegraphics[width=\fw]{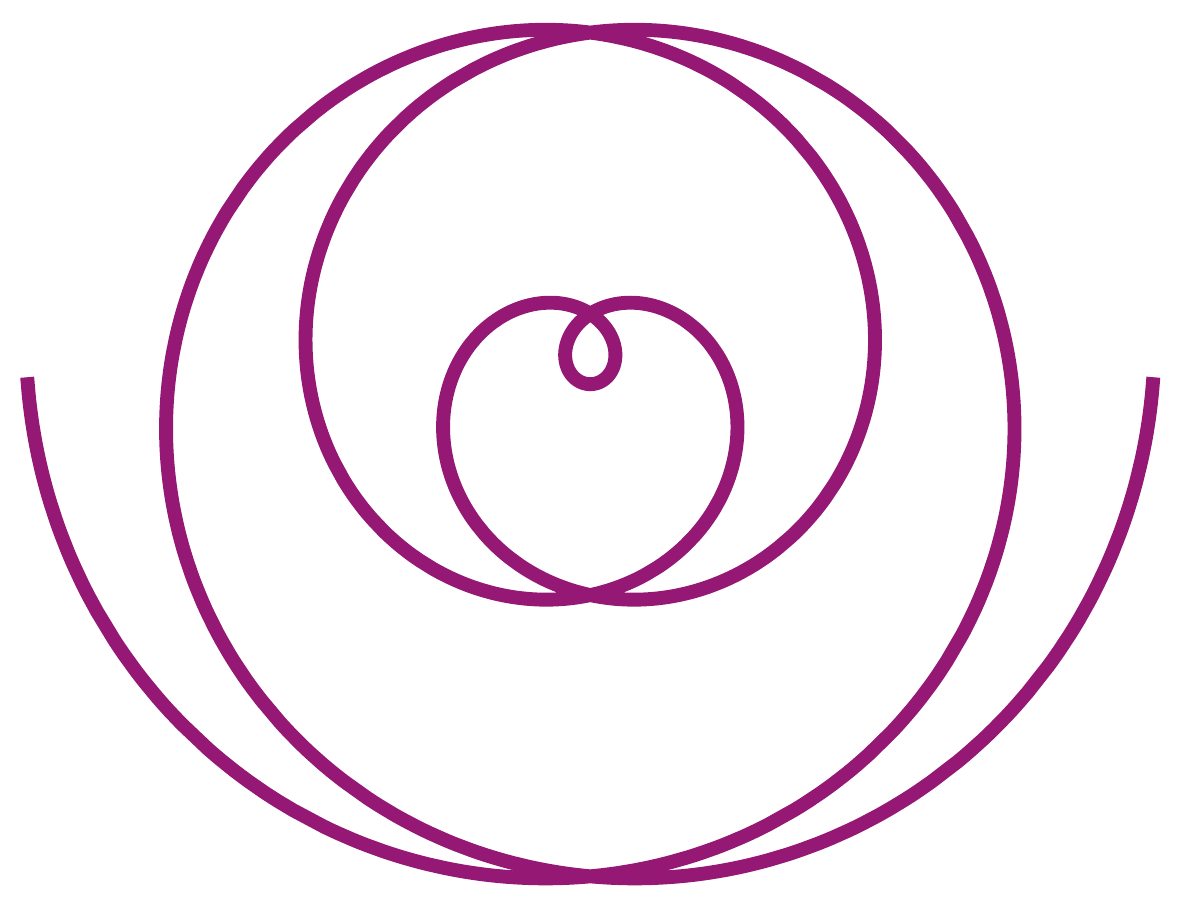}
  		\caption{Archimedean Spiral}
	\end{subfigure}
 	 \caption{Two Spirals}
 	 \label{fig:planespirals}
  \end{figure}

As examples for non-trigonometric polyexp curves, we consider logarithmic and Archimedean spirals.

Logarithmic spirals are given by

 \[
(x(t),y(t)) = \rho^t( \cos (t), \sin (t)) \ .
\]

Here, $z(t)$ becomes
\[
z(t) =  \frac12 {\lambda  t}-\frac{\left(\log ^2(\rho )+1\right) }{4 \lambda  \log
   (\rho )}\rho ^{2 t}  \ .
\]
Observe that the linear term in $t$ guarantees that the space curve becomes proper, see Figure \ref{fig:logspir}.

\def\fw{2.7in}
\begin{figure}[H]
	\centering
	\begin{subfigure}[t]{\fw}
  		\centering
		\includegraphics[height=\fw]{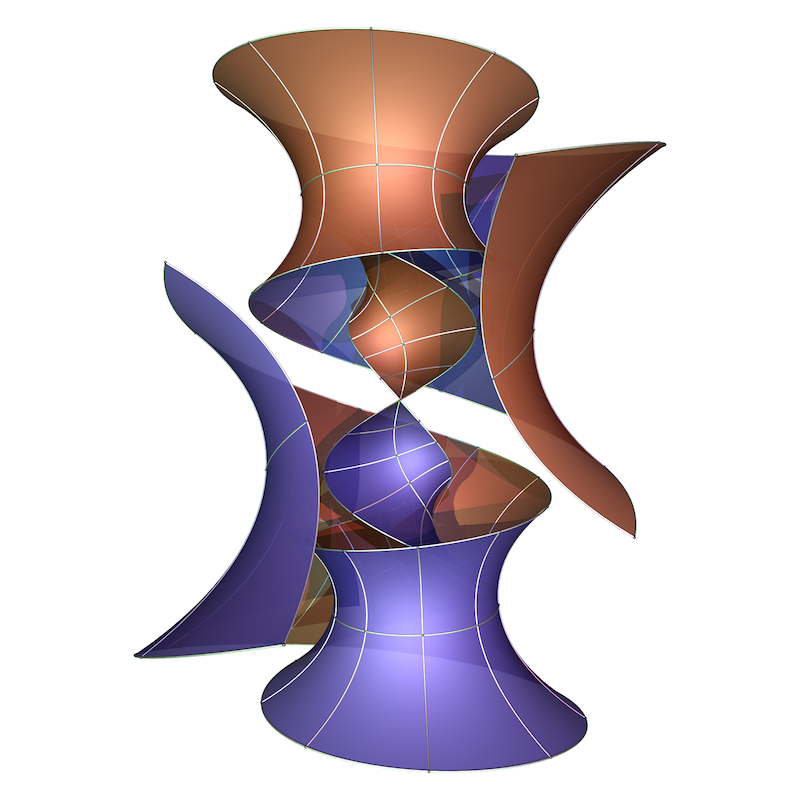}
  		\caption{Archimedean Spiral}
		\label{fig:archimedes}
	\end{subfigure}
	\quad
	\begin{subfigure}[t]{\fw}
  		\centering
		\includegraphics[height=\fw]{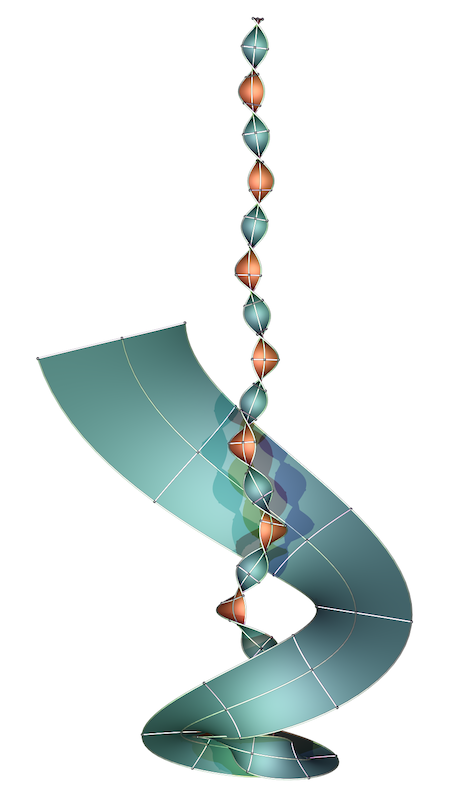}
  		\caption{Logarithmic Spiral}
		\label{fig:logspir}
	\end{subfigure}
 	 \caption{\bj surfaces based on a spirals}
 	 \label{fig:spirals}
  \end{figure}

For the Archimedean spiral
\[
(x(t), y(t)) = t(\cos(t), \sin(t))
\]
we obtain the cubic polynomial
\[
z(t) = -\frac1{6\lambda} t \left(t^2 +3-3\lambda^2\right) \ .
\]
This leads for $|\lambda|>1$ to core curves whose $z$-coordinate has two local extrema,
well visible in Figure \ref{fig:archimedes}.

\def\fw{2.5in}
\begin{figure}[h]
	\centering
	\begin{subfigure}[t]{\fw}
  		\centering
		\includegraphics[width=\fw]{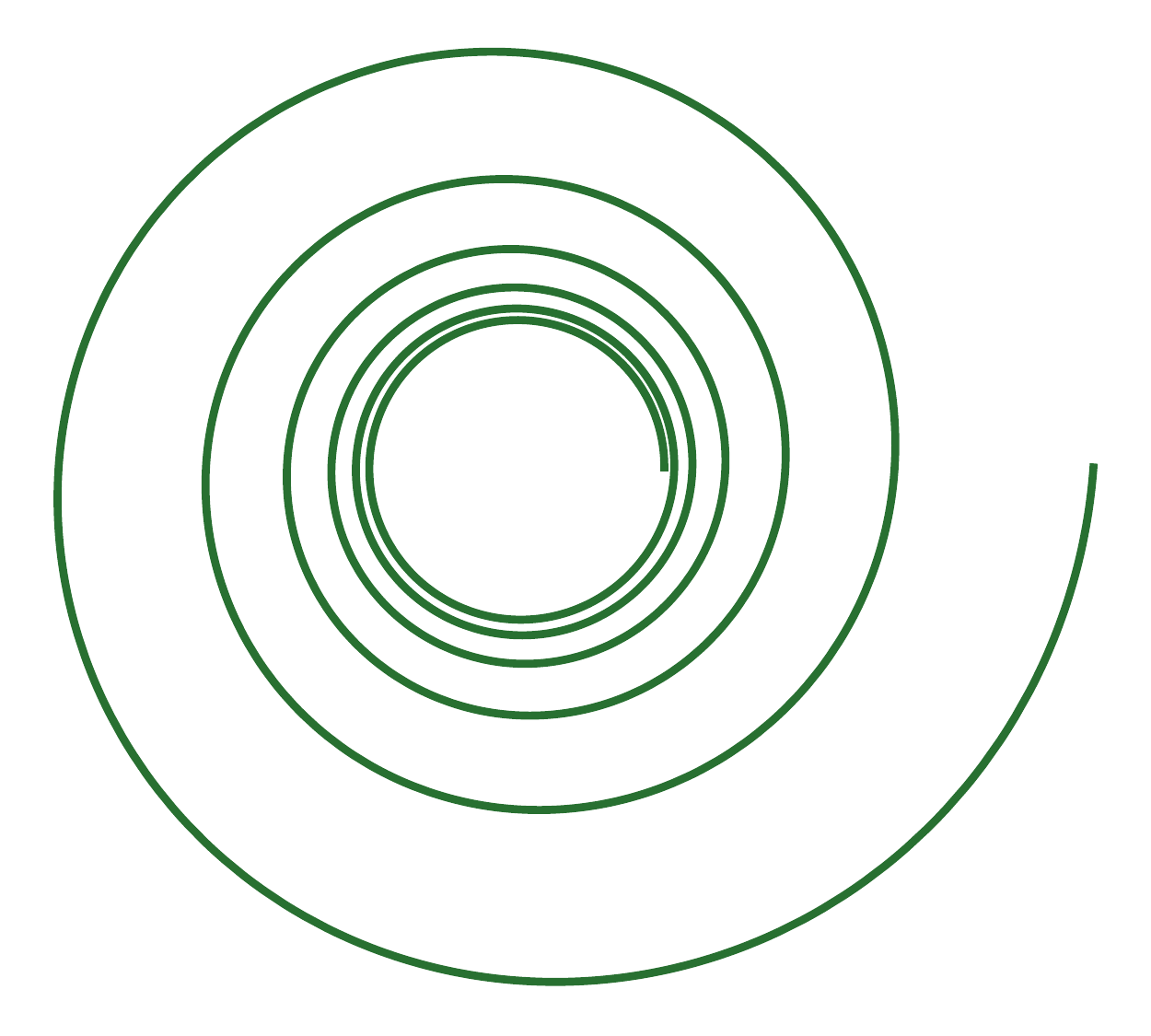}
  		\caption{Spiral about a circle}
		\label{fig:circlespiral}
	\end{subfigure}
	\quad
	\begin{subfigure}[t]{\fw}
  		\centering
		\includegraphics[width=\fw]{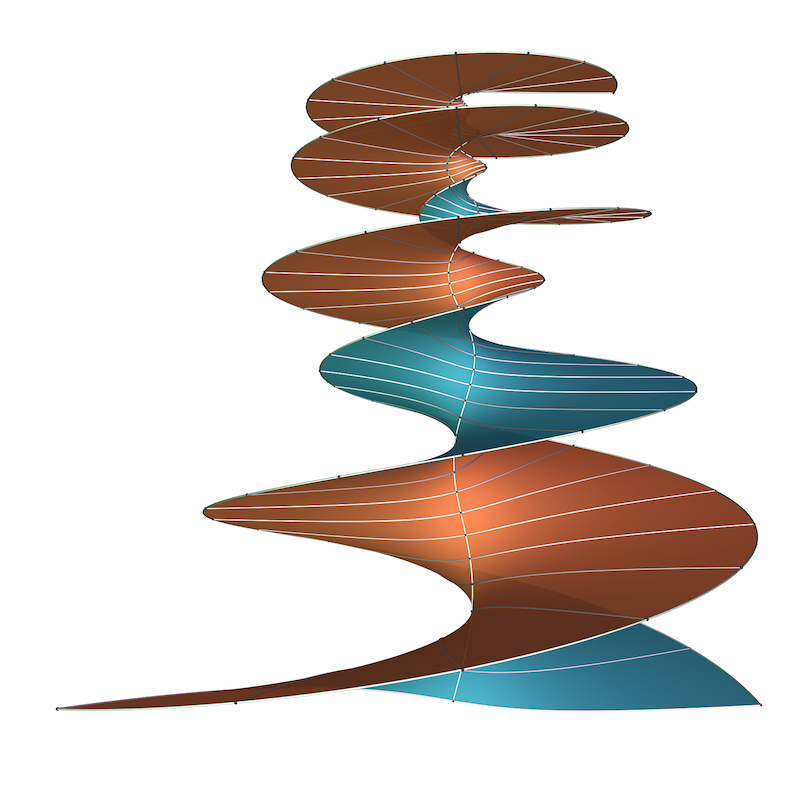}
  		\caption{Helicoid with limit leaf}
		\label{fig:limitleaf}
	\end{subfigure}
 	 \caption{\bj surface using a lifted curve with bounded height}
 	 \label{fig:limits}
  \end{figure}

As a variation of the logarithmic spiral, let   

\[
(x(t), y(t)) =(\rho^t+1) (\cos(t), \sin(t))
\]
which limits for $t\to -\infty$ on the unit circle (see Figure \ref{fig:circlespiral}). We obtain the lift
\[
z(t) = -\frac1{4\lambda \log(r)}  \left((1+\log^2(\rho)) \rho^{2t}+ 4 \rho^t -(\lambda^2-1) t\right) \ .
\]
 This means that for $\lambda=1$, the lifted curve will for $t \to -\infty$  limit on the unit circle at height 0. If we choose in addition $a=3$ and $b=\frac\pi2$, the closure of the surface becomes a minimal lamination in a cylinder about the vertical axis with two leaves: One is the \bj surface (see Figure \ref{fig:limitleaf}), the other a horizontal disk at height 0. This example is very similar to the one constructed in \cite{cm28}.
 
 \subsection{The Clothoid}
  
\def\fw{2.5in}
\begin{figure}[H]
	\centering
	\begin{subfigure}[t]{\fw}
  		\centering
		\includegraphics[width=\fw]{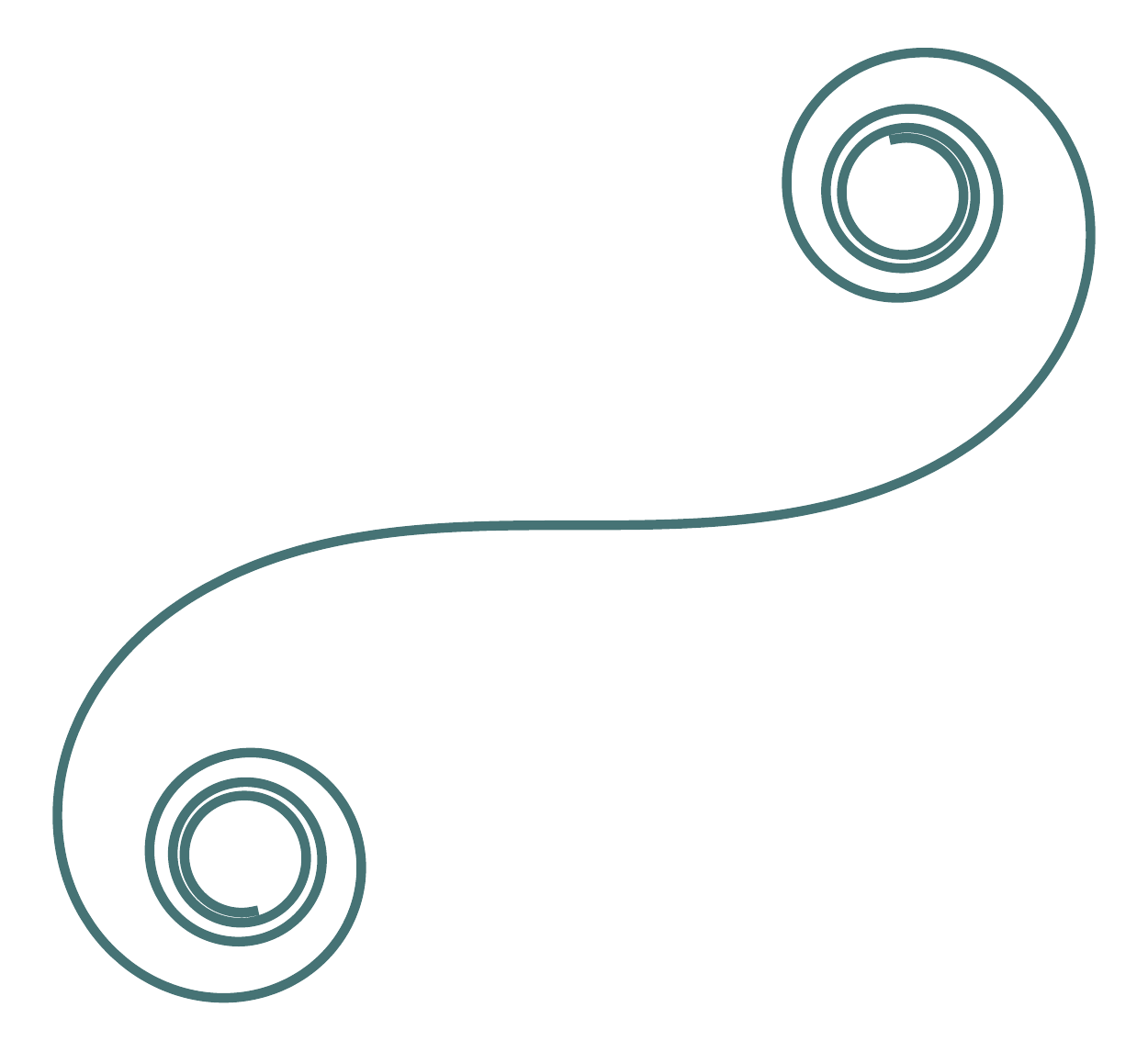}
  		\caption{The clothoid}
			\end{subfigure}
	\quad
	\begin{subfigure}[t]{\fw}
  		\centering
		\includegraphics[width=\fw]{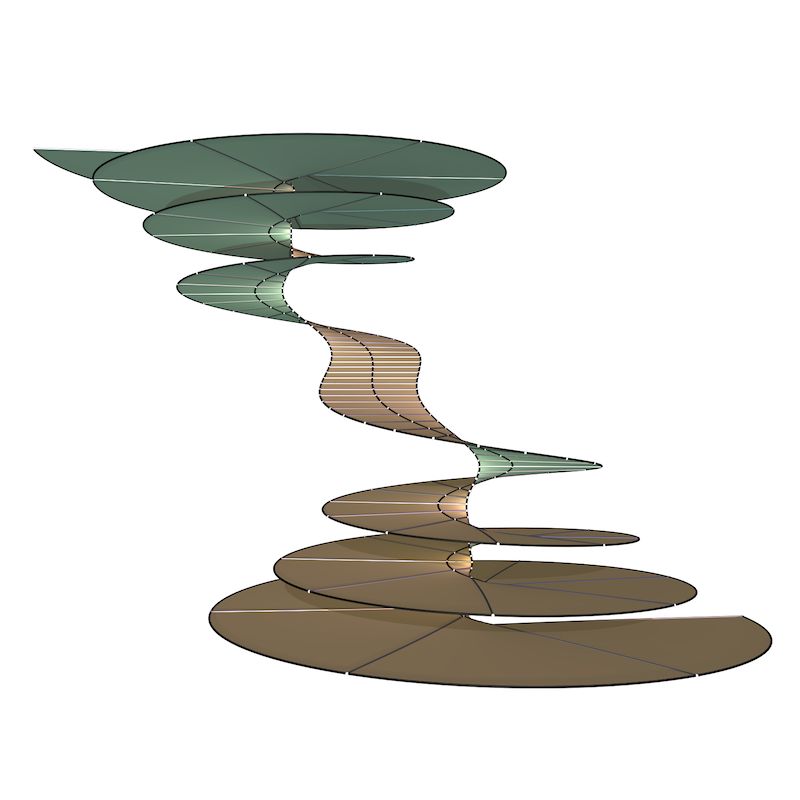}
  		\caption{$\lambda = 1.4$}
	\end{subfigure}
 	 \caption{\bj surface based on a lifted clothoid}
 	 \label{fig:clothoid}
  \end{figure}

In this last example, we discuss the possibility to design \bj surfaces based on curves that are not polyexp. 
  Consider the clothoid given by
  
  \[
  (x(t), y(t)) = (C(t), S(t))
  \]
  where
  \begin{align*} 
     C(t) = {}& \int_0^t \cos(s^2)\, ds  \\
     S(t) = {}& \int_0^t \sin(s^2)\, ds  \\
  \end{align*}
  are the Fresnel integrals.
  
  The $z$-coordinate of the lift is given by
  \[
  z(t) = \frac{\lambda^2-1}{2\lambda} t \ .
  \]
  
  As rotating normal we choose
  \begin{align*} 
  n(t) ={}& \cos(t^2) n_1(t) + \sin(t^2) n_2(t) \\
  ={}& \frac1{\lambda ^2+1}
  \begin{pmatrix}
   {\left(1-\lambda ^2\right) \cos \left(t^2\right)} \\
 {\left(1-\lambda ^2\right) \sin \left(t^2\right)} \\
{2 \lambda } \\
 \end{pmatrix} \ .
  \end{align*}
  Note that we are adapting the rotational speed to the parametrization of the clothoid. This results in very simple \we data
    \begin{align*} 
     G = {}&\frac{1+\lambda}{1-\lambda}  e^{i z^2} \\
     dh = {}& \frac{\lambda^2-1}{2\lambda}\, dz  \\
  \end{align*}
  and in the almost horizontal \bj surface in Figure \ref{fig:clothoid}. The only non-elementary functions in the surface parametrization are the Fresnel integrals:
  
  \[
  f(z) = \frac12 \re
  \begin{pmatrix}
  2 \lambda  C(z)+i \left(\lambda ^2+1\right) S(z)\\
  -i \left(\lambda ^2+1\right) C(z)+2  \lambda  S(z) \\
 z \left(\lambda ^2-1\right)  \ .
  \end{pmatrix} 
  \]

\bibliographystyle{plain}
\bibliography{bibliography}

\end{document}